\documentclass[a4paper,aos,noinfoline,authoryear]{imsart}
\usepackage{amssymb, amsmath,amsthm,mathrsfs,fmtcount,url,bm,nicefrac}
\usepackage{multirow}
\usepackage{graphicx,graphics,psfrag}
\usepackage[usenames,dvipsnames]{color}
\usepackage[comma,sort&compress]{natbib}
\RequirePackage[OT1]{fontenc}

\newtheorem{thm}{Theorem}[section]
\newtheorem{lem}{Lemma}[section]
\newtheorem{cor}{Corollary}[section]

\newcommand{\var}{\ensuremath{{\mathbb Var}}}
\newcommand{\ex}{\ensuremath{{\mathbb E}}}

\newcommand{\E}{\ensuremath{\text{E}}}

\newcommand{\X}{\ensuremath{\mathbf X}}
\newcommand{\Y}{\ensuremath{\mathbf Y}}
\newcommand{\x}{\ensuremath{\mathbf x}}
\newcommand{\y}{\ensuremath{\mathbf y}}

\newcommand{\phat}{\ensuremath{\widehat{p}}}
\newcommand{\Shat}{\ensuremath{\widehat{\Sigma}}}

\newcommand{\thetab}{\ensuremath{\bm{\theta}}}
\newcommand{\that}{\ensuremath{\widehat{\theta}}}

\newcommand{\chat}{\ensuremath{\widehat{c}}}
\newcommand{\dhat}{\ensuremath{\widehat{d}}}
\newcommand{\Dhat}{\ensuremath{\widehat{D}}}
\newcommand{\Uhat}{\ensuremath{\widehat{U}}}

\newcommand{\mse}{\ensuremath{ \text{MSE\,}}}

\newcommand{\U}{\ensuremath{ \text{U}}}
\newcommand{\citeasnoun}{\citet}
\newcommand{\smallo}{\ensuremath{\text{o}\,}}
\newcommand{\bigo}{\ensuremath{\text{O}\,}}
\newcommand{\I}{\mathbb{I}}

\newcommand{\loss}{\textbf{L}}

\newcommand{\RR}{\mathbb{R}}

\newcommand{\idealf}{\ensuremath{\text{IF}}}
\newcommand{\bias}{\ensuremath{\text{Bias}}}
\newcommand{\std}{\ensuremath{\text{SD}}}

\newcommand{\sskip}{\\[1ex]}
\newcommand{\ssskip}{\hspace{.1cm}\\[1ex]}
\newcommand{\mpar}{\sskip \indent}

\begin{document}

\bibliographystyle{imsart-nameyear}
\begin{frontmatter}
\title{On the within-family Kullback-Leibler risk in Gaussian Predictive $\text{models}^{\dagger}$}
\runtitle{Optimal Gaussian Predictive Risk}
\footnotetext{$\bm{\text{\hspace{0.01cm}}^{\dagger}}$This paper is based on Chapter 3 of G.M.'s Ph.D. thesis (to be submitted) and was titled as `On efficient quadratic approximations of the predictive log-likelihood in a Gaussian sequence model' in previous preprints.\\[-2.5ex]}

\begin{aug}
\author{\fnms{Gourab} \snm{Mukherjee}\ead[label=e1]{gourab@stanford.edu}}
\and
\author{\fnms{Iain M.} \snm{Johnstone}\ead[label=e2]{imj@stanford.edu}}
\affiliation{Stanford University}

\address{Address of the First and Second authors\\
Department of Statistics\\ 
Sequoia Hall, 390 Serra Mall\\
Stanford University\\ 
Stanford, CA 94305-4065\\ 
\printead{e1}\\
\phantom{E-mail:\ }\printead*{e2}}
\end{aug}

\begin{abstract}
We consider estimating the predictive density under Kullback-Leibler loss in a high-dimensional Gaussian model. 
Decision theoretic properties of the within-family prediction error -- the minimal risk among estimates in the class $\mathcal{G}$ of all Gaussian densities are discussed. We show that in sparse models, the class $\mathcal{G}$ is minimax sub-optimal.
We produce asymptotically sharp upper and lower bounds on the within-family prediction errors for various subfamilies of $\mathcal{G}$.
Under mild regularity conditions, in the sub-family where the covariance structure is represented by a single data dependent parameter $\Shat=\dhat \cdot I$, the Kullback-Leiber risk has a tractable decomposition which can be subsequently minimized to yield optimally flattened predictive density estimates. The optimal predictive risk can be explicitly expressed in terms of the corresponding mean square error of the location estimate, and so, the role of shrinkage in the predictive regime can be determined based on point estimation theory results.
Our results demonstrate that some of the decision theoretic parallels between predictive density estimation and point estimation regimes can be explained by second moment based concentration properties of the quadratic loss.  
\end{abstract}

\begin{keyword}[class=AMS]
\kwd[Primary ]{62C20}
\kwd[; Secondary ]{62M20}
\kwd{60G25}
\kwd{91G70}
\end{keyword}

\begin{keyword}
\kwd{prediction}
\kwd{predictive density}
\kwd{risk diversification}
\kwd{shrinkage}
\kwd{sparsity}
\kwd{high-dimensional}
\kwd{Kullback-Leibler}
\kwd{moment-based concentration inequalities}
\end{keyword}
\end{frontmatter}

\section{Introduction and main result}
\subsection{Background}
We consider a prediction set-up where the observed past data $\X$ and the unobserved future data $\Y$ are generated from a joint parametric density 
$f_{\thetab}(\x,\y)$ where $\thetab$ is the unknown parameter. A perspective in prediction analysis is to use the concept of predictive likelihoods \citep{Hinkley79,Lauritzen74} and its variants \citep{Bjornstad90}, to infer about the future $\Y$ based on $\X$, with $\thetab$ playing the role of a nuisance parameter. Most predictive likelihoods \citep{Butler86} are functions of the future conditional density $f_{\thetab}(\y\,|\X=\x)$ which is also referred to as the predictive density \citep{Geisser71}. Efficient estimates of the predictive density will ensure good predictive performances.
Here, we study the problem of predictive density estimation in a high-dimensional Gaussian model.   
\mpar

We consider the multiple regression model analyzed in \citet{George08}. Suppose, the observed past $\X$ is independently generated from $m_1$ dimensional product Gaussian density  $N(A\,\thetab, \sigma_p^2 \, I)$ indexed by an $n$--dimensional unknown parameter $\thetab$ and known variance $\sigma_p^2$ and known $m_1 \times n$ data-matrix $A$. 
The future $Y$ is generated from the $m_2$--dimensional Gaussian density $N(B\,\thetab, \sigma_f^2\, I)$ with the time-invariant parameter $\thetab$ and 
known  $m_2 \times n$ dimensional matrix $B$ and known future volatility $\sigma_f^2$.\medskip\\
\textbf{Homoscedastic Gaussian Predictive Model}
$$\textbf{M.1}\qquad \X \sim {N}(A\,\thetab, \, \sigma_p^2 \, I) \; \; \text{ and } \; \;\Y \sim {N}({B} \, \thetab, \, \sigma_f^2 \, I)$$
The location structure depend on the time-invariant unknown vector $\thetab$ of length $n$. If $\thetab$ is fixed the true predictive density of $\Y$ would be $p_{(\thetab, B)}(\cdot)=N(B\,\thetab,\sigma_f^2\,I)$. We would like to estimate it by density estimates $\phat\,(.|\X=\x)$. 
\par
We use the information theoretic measure of \citet*{Kullback51} as the goodness of fit measure between the true and estimated distributions 
$$ \loss\,\big(\,\thetab, \,\phat\,\big(\,\cdot\,\big|\x)\,\big)= \int p_{(\thetab,B)}(\y) \; \log \left( \frac{p_{(\thetab,B)}(\y)}{\phat\big(\y\big|\x\big)}\right) \; d\y\;.$$
Averaging over the past observations $\X$, the predictive risk of the density estimate $\phat(\cdot\,\vert \X=\x)$ at $\thetab$ is given by 
\begin{align}\label{p.risk}
\bm{\rho}\,\big(\,\thetab,\, \phat\,\big)= \iint \, p_{(\thetab,A)}(\x) \; p_{(\thetab,B)}(\y) \log \left( \frac{p_{(\thetab,B)}(\y)}{\phat\big(\y\big|\x\big)}\right) \; d\y \; d\x\; . 
\end{align}
The relative entropy predictive risk $\bm{\rho}\,\big(\,\thetab,\, \phat\,\big)$ measures the exponential rate of divergence of the joint likelihood ratio over a large number of independent trials \citep{Larimore83}. The minimal predictive risk estimate maximizes the expected growth rate in repeated investment scenarios \citep[Chapter 6 and 15]{Cover-book}. Competitive optimal predictive schemes \citep{Bell80} for gambling, sports betting, portfolio selection, etc can be constructed from predictive density estimates with optimal Kullback-Leibler (KL) risk properties. 
Our Gaussian predictive framework can accommodate a fairly large number of prediction scenarios as often, in high-dimensional models, good normalization transformations of the data are available \citep[Chapter 1, Page 8]{Efron-book}. In the data compression set-up $\loss\,\big(\,\thetab, \,\phat\,\big(\,\cdot\,\big|\x)\,\big)$ reflects the excess average code length that we need in Gaussian channels if we use the conditional density estimate $\phat$ instead of the true density to construct a uniquely decodable code for the data $\Y$ given the past $\x$ \citep{Mcmillan56}. The notion can be extended to a sequential framework where minimizing the predictive risk would result in the minimum description length \citep{Rissanen84,Barron98} based estimate of the true parametric density \citep{Liang05}.
\par
Here, we discuss efficient estimators in the class $\mathcal{G}_n$ of all $n$--dimensional Gaussian distributions with positive definite (p.d.) covariances as the dimension increases, i.e.,
$$\mathcal{G}=\bigg\{ g:\RR^n \to \RR^+ \text{ such that } g = N(\mu,\Sigma) \text{ where } \mu \in \RR^n \text{ and } \Sigma \text{ p.d.}\bigg\}.$$
For any class $\mathcal{C}$ we define the predictive risk of the class as
$$\bm{\rho}_{\mathcal{C}}(\thetab)=\inf_{\phat \in \mathcal{C}} \, \bm{\rho}(\thetab,\phat).$$ 
As the true parametric density is also Gaussian, $\bm{\rho}_{\mathcal{G}}(\thetab)$ represents the within-family predictive risk.
We also evaluate  the predictive risk of the sub-family $\mathcal{G}[p]$ which contains all product Gaussian densities.
We also make inferences in sparsity restricted parameter spaces. We impose an $\bm{\ell_0}$ constraint on the parameter space: 
\begin{align}\label{sparse.parametric.space}
\Theta(n,s)=\left\{\thetab \in \RR^n : \sum_{i=1}^n \I[\theta_i \neq 0] \leq s \right\}.
\end{align}
This notion of sparsity is widely used in modeling highly interactive systems (represented by a large number of related parameters) which are dominated by only few significant effects. Sparse models have been successfully employed in biological sciences \citep{Tibshirani02}, engineering applications \citep{Donoho04} and financial modeling \citep{Brodie09}. The predictive model $\textbf{M}$ with $\bm{\ell_0}$ constraint on the location structure can be used for sparse coding and for prediction in sparse networks.    
\par
As in point estimation, risk calculations in $\textbf{M}$ would intrinsically depend on risk calculations in the orthogonal model: \\[1.5ex]
\textbf{Orthogonal Gaussian Predictive Model}
\begin{align*}
\textbf{M.2}\qquad & \X \sim {N}(\thetab, \, \sigma_p^2 \, I) \; \; \text{ and } \; \;\Y \sim N(\thetab, \, \sigma_f^2 \, I)
\end{align*}
where $\X$ and $\Y$ are both $n$ -- dimensional vectors. Most of our calculations will be in high-dimensions (which means $n \to \infty$ in the orthogonal model) though dimension independent bound will also be provided. As $n \to \infty$, \textbf{M.2}  represents the Gaussian sequence model
\citep{Nussbaum96} and has been widely studied in the function estimation framework \citep{Johnstone-book}. Estimation in $\textbf{M.1}$ can be linked with the decision theoretic results in $\textbf{M.2}$ through the procedure outlined in \citet*{Donoho11}. 

\subsection*{Our Contributions} Efficacy of predictive density estimates has been a subject of considerable interest in predictive inference. 
\cite{Aitchison75,Aslan06,Komaki96,Hartigan98} determined asymptotically optimal (admissible) Bayes predictive density estimates in fixed dimensional parametric family whereas minimax optimality in restricted parameter spaces has been discussed in \cite{Fourdrinier11} and \cite{Kubokawa12}.
Recently, \cite{George06,Brown08,Ghosh08} extended the admissibility results to high dimensional Gaussian models.
However, the optimal estimates are not necessarily Gaussian and using them in high-dimensional problems would involve computationally intensive methods.
Here, we find optimal predictive density estimates within the Gaussian family and also compute their predictive risk.
It is computationally easier to construct predictive attributes based on our optimal Gaussian predictive density estimates  and the optimal Gaussian predictive risk assures guaranteed performances of our strategies.
\par
Minimizing the Gaussian predictive risk involves simultaneous estimation of the location and scale parameters. The issue of joint estimation of location and scale (and to a degree the shape) has not been addressed before in one sample Gaussian models. However, separate estimation of location
\citep{Tibshirani11} and covariance \citep{Friedman08} are well-studied topics in constrained Gaussian estimation.
Also, as reviewed in \cite{George12} decision theoretic parallels exist between point estimation theory under quadratic loss and predictive density estimation under Kullback-leibler loss in high-dimensional Gaussian models. Here, our results demonstrate that some of these decision theoretic parallels (in the class $\mathcal{G}$) can be explained by second moment based concentration properties on the quadratic loss of location point estimators in high dimensions. The moment based approach used here for estimating the scale parameter bears resemblance to concepts seen elsewhere in prediction theory, particularly in the the theory of cross validation \citep{Yang07} and covariance penalties for model selection \citep{Efron04,Ye98} .

\subsection{Description of the main results}
We describe results of two different flavors concerning
(i)  the minimax risk of the class $\mathcal{G}$ of Gaussian density estimates in sparsity restricted spaces 
(ii) the optimality (asymptotic admissibility, oracle inequality, risk upper bounds) of estimates in class $\mathcal{G}$ in unrestricted space (over $\RR^n$ as $n \to \infty$) where the role of shrinkage comes into play.  
In order to describe the results, we need to introduce the following notations.
\subsection*{Notation and Preliminaries}
As some of our results are dimension dependent, henceforth we refrain from using bold representation for vectors and denote the dimension in the subscript.  
Given any fixed sequence $\theta_{\infty}$ we represent the first $n$ values by the $n$--dimensional vector $\theta_n$ whereas $\theta(n)$ denotes  the $n^{\text{th}}$ value, i.e $\theta_{n+1}=(\theta_n,\theta(n+1))$.
By $\mathcal{G}_n[p]$ we denote the class of all $n$--dimensional product Gaussian densities
$$\mathcal{G}_n[p]=\bigg\{ g[\mu_n,D_n]: \mu_n \in \RR^n \text{ \& } D_n \text{ is any $n \times n$ p.d. diagonal matrix}\bigg\}$$
where $g[\mu_n,D_n]$ is a normal density with mean $\mu_n$ and diagonal covariance $\sigma_f^2 D_n$. We represent the minimal Gaussian predictive risk by $\rho_{\mathcal{G}}(\theta_n):=\inf_{\phat \in \mathcal{G}_n} \rho(\theta_n,\phat)$.\par
Our shrinkage results will mostly refer to the sub-family $\mathcal{G}_n[1]$ of $\mathcal{G}_n[p]$.  $\mathcal{G}_n[1]$ contains Gaussian densities with only one data-adaptive scale estimate
$$\mathcal{G}_n[1]=\bigg\{ g[\mu_n,c]: \mu_n \in \RR^n \text{ and } c \in \RR^+\bigg\}.$$  
where $g[\that_n,c]$ denotes a normal density with mean $\mu_n$ and covariance $c \, \sigma_f^2 \, I$. 
A typical density estimate in $\mathcal{G}_n[1]$ is represented as $g[\that_n,\chat\,(n)]$ where  $\that_n$ is a location estimate and $\chat(n)$ is the scale estimate based on observing an $n$--dimensional past observation $X_n$. For any fixed location estimate $\that_n$, the optimal risk of density estimates in $\mathcal{G}_n[1]$ centered around $\that_n$ is given by
$$\rho_0(\theta_n,\that_n)=\inf_{\chat(X_n) \in \RR^+} \rho(\theta_n,g[\that_n,\chat(X_n)]).$$ 
The quadratic risk of the location estimate is denoted by $$q(\theta_n,\that_n)=\ex_{\theta_n}\Vert \that\,(X_n)-\theta_n \Vert^2$$
where the expectation is over the observed past $X_n$.
Later, we show that if the value of $q(\theta_n,\that_n)$ were known, then the optimal choice for scale is
$$\idealf_{\theta_n}(\widehat{\theta}_n)= 1+n^{-1}r^{-1}q(\theta_n,\widehat{\theta}_n)$$
which will be called as the Ideal Flattening coefficient for $\that_n$ at $\theta_n$. 
Here, given a location estimate $\that_n$ we construct suitable estimates $\chat(n)$ of the scale such that asymptotically when $n \to \infty$ the  density estimate  $g[\that_n,\chat(n)]$ is optimally flattened in the sense that $\rho(\theta_n,g[\that_n,\chat(n)]) - \rho_0(\theta_n,\that_n) \leq \bigo(1)$.
However, for proving optimality of the flattening coefficient we need the following mild regularity conditions on the location estimate $\that_n$:
\begin{align}\label{rasl.1.eqn}
&\quad q(\theta_n,\that_n)\leq \bigo(n).\text{\hspace{7cm}}\\
&\quad \var_{\theta_n\,}( \,\Vert \,\widehat{\theta}_n - \theta_n \,\Vert^2 \,) \leq \bigo(n)
\end{align}
and the existence of a suitable estimate $\U[\that](X_n)$ for the quadratic risk of $\that_n$ at $\theta_n$ with the following properties:
\begin{align}
&\quad\left | \, \ex_{\theta_n} \big(\widehat{U}_n\big) -  q(\theta_n,\that_n) \, \right | \leq \bigo(\,n^{1/2}\,). \text{\hspace{4cm}}\label{rasl.2.eqn.1}\\
&\quad\var_{\theta_n}\big(\widehat{U}_n\big) \leq \bigo(n).\label{rasl.2.eqn.2}\\
&\quad\var_{\theta_n}\left\{\left(1+(nr)^{-1}\widehat{U}_n\right)^{-1}\right\}\leq \bigo(n^{-1}).\label{rasl.2.eqn.3}
\end{align} 
These properties are fairly mild and in Section 2 we show that most popular point estimators obey these above conditions.
We call these conditions Reasonable Asymptotic Square Loss (RASL) properties and the set of point estimators in the sequence model (where the action set is $\RR^{\infty}$) which satisfies these conditions is denoted by $\mathcal{A}$.    
Also, we denote the ratio of the future to past variances by $r:=\sigma_f^2/\sigma_p^2$. Our results will depend on $r$. For sequences, the symbol $a_n \sim b_n$ means $a_n = b_n(1 + \smallo (1))$ and $a_n \approx b_n$ means $a_n/b_n \in (k_1, k_2)$ where
$k_1$ and $k_2$ are constants.
\paragraph{Results} 
We show that in high dimensions, the minimum predictive entropy risk of Gaussian density estimates around reasonable location estimate $\widehat{\theta}_n$ can be expressed in terms of the corresponding quadratic risk of $\widehat{\theta}_n$. The minimum predictive risk can be attained by optimally flattening the normal density estimate around $\that_n$. 
The choice of the optimal flattening coefficient is not unique. An asymptotically efficient choice based on a reasonable estimate $U[\widehat{\theta}_n](X_n)$  of the quadratic risk of $\widehat{\theta}_n$ can be made. 
\begin{thm}\label{rasl.main.thm}
For any estimator $\that$ in $\mathcal{A}$ we have
\begin{align} 
\bigg \vert \rho_0(\theta_n,\that_n) - \frac{n}{2} \log \left ( 1+ (nr)^{-1}\, \cdot \,
q(\theta_n,\that_n) \right) \, \bigg \vert
 \leq \bigo\big(1\big) \; \text{ as } n \to \infty.
\end{align}
And if $ \widehat{c}\,(X_n)=1+(nr)^{-1} U[\,\widehat{\theta}\,](X_n)$ is based on  a suitable estimate $\widehat{U}_n$ of the quadratic risk as defined in Equations~(\ref{rasl.2.eqn.1})--(\ref{rasl.2.eqn.3}) then 
\begin{align}
\rho(\theta_n,g[\that_n,\chat(n)]) - \rho_0(\theta_n,\that_n) \leq \bigo(1).
\end{align}
\end{thm}
We represent the optimal density estimate $g[\that_n,\chat(n)]$ by $g[\that_n]$.
Based on the asymptotic relations between $\rho_0(\theta_n,\cdot)$ and the Mean Square Error (MSE) $q(\theta_n,\cdot)$, we can characterize the predictive risk of $g[\,\widehat{\theta}_n]$ easily by plugging in standard oracle inequalities from point estimation theory. 
We check that RASL conditions defined in Equations~(\ref{rasl.1.eqn})--(\ref{rasl.2.eqn.3}) hold for the James-Stein estimator \citep{Stein81} 
$$\widehat{\theta}^{\,JS}_n= X_n\left(1- \frac{n-2}{\Vert X_n\Vert^2}\right)$$
and its positive part estimator $\widehat{\theta}^{\,JS+}$. For the James-Stein estimator we determine the deviations from the optimal risk in terms of dimension dependent bounds.
\begin{thm}\label{thm.dim.indep}
For any dimension $n \geq 10$ and for any $\theta_n \in \RR^n$ we have,
$$\rho\big(\theta_n,g[\that_n^{\,JS}]\big) - \rho_0\big(\theta_n,\that_n^{\,JS}\big)\leq 2^{-1}\big( a_n^{1/2}\, b_n^{1/2}\,r^{-3/2}\,+(a_n+b_n+l_n)\,r^{-2}+a_n\,r^{-3}\big)$$
where the constants $a_n,b_n,l_n$ are independent of the parameter but depend on the dimensions $n$ and are given by
\begin{align*}
&a_n=3\big(1-(n-2)^{-1}\big)^{-2},\;b_n=4(2+a_n+k_2(n)),\;\\
&l_n=3(1-2/n)^{-2},\; k_2(n)=\max \{ e(n),f(n) \}   \text{ with } \\
&e_n = \sqrt{3}\;{\prod_{i=1}^4 (1-(2i+1)/n)\}^{-1/2}} \text{ and } f_n=(1- (\log n / n)^{1/2})^{-2}.
\end{align*}
Also, $\rho\big(\theta_n,g[\that_n^{\,JS}]\big)$ can be approximated by using the following bound
\begin{align*}
\bigg\vert \rho\big(\theta_n,g[\that_n^{\,JS}]\big) - \frac{n \log \idealf_{\theta_n}(\that_n)}{2} \bigg\vert \leq \frac{\big( a_n^{1/2}\, b_n^{1/2}\,r^{-3/2}\,+(a_n+b_n)\,r^{-2}+a_n\,r^{-3}\big)}{2}.
\end{align*}
\end{thm}
These bounds hold for any value of $r \in (0,\infty)$. As the ratio of the future to  past variances $r$ decreases, we need to estimate the future observations based on increasingly noisy past observations and so, the difficulty of the density estimation problem also increases.
So, as expected when $r$ decreases the bounds also increases. These bounds can be made dimension independent.\\
In particular, for all dimension $n \geq 20$, for any $\theta_n \in \RR^n$ and for any fixed value of $r \in (0,\infty)$, we have,
\begin{align}
\rho\big(\theta_n,g[\that_n^{\,JS}]\big) - \rho\big(\theta_n,\that_n^{\,JS}\big)\leq 5.3\,r^{-3/2} + 19.6\, r^{-2} + 1.7 r^{-3}.
\end{align}  
In point estimation theory, there exist sharp oracle bounds on the quadratic risk $q(\theta_n,\that^{\,JS}_n)$ of the James-Stein estimator $\that^{\,JS}$\citep[Chapter 2]{Johnstone-book} which along with Theorem~\ref{thm.dim.indep} produce the following oracle bound on the predictive risk of shrinkage predictive density estimates. Assuming that the value $||\theta_n||^2$ is known, the risk of the ideal linear predictive density estimate is given by
\begin{align}\label{ideal.linear.risk}
\text{IL\,}(\theta_n)=\frac{n}{2}\, \log \bigg(1+r^{-1}\,\frac{ a_n}{1+a_n}\bigg) \text{ where } a_n=||\theta_n||^2/n.
\end{align}
The difference in the risk of $g[\that_n^{\,JS}]$ and the optimal oracle linear risk is 
\begin{align}\label{oracle.bound.1}
\rho\big(\theta_n,g[\that_n^{\,JS}]\big) -\text{IL\,}(\theta_n) \leq 0.1\,r^{-1} + 5.3\,r^{-3/2} + 18.1\, r^{-2} + 1.7 r^{-3}.
\end{align}
Comparing this to the oracle bound of \cite{Xu12} which is derived based on an empirical Bayes perspective
\begin{align}\label{oracle.bound.2}
\rho\big(\theta_n,g[\that_n^{\,JS}]\big) - \text{IL\,}(\theta_n) \leq 2 \,r^{-1} + 5\, r^{-2} + 4 r^{-3},
\end{align}
the particular features of our moment based approach can be seen. 
As our oracle inequality is a by-product of the optimal Gaussian risk, for most values of $r$ the bound in the Inequality~(\ref{oracle.bound.1}) is coarser than that in Inequality~(\ref{oracle.bound.2}). 
However, when $r=0.1$, the RHS in the Inequality~(\ref{oracle.bound.1}) is $3830$ and is better than the bound (4520) in the latter.
Thus, the moment based approach can be quite informative. The bounds derived on the predictive risk are sharp enough to derive decision-theoretic optimality. We can produce unrestricted improved minimax predictive densities which asymptotically behaves like ideally shrunk linear density estimates (as defined later). The following lemma shows the asymptotic improvement in the predictive risk over the best invariant predictive density $g[X_n,1+r]$ \citep{Liang04}. 
\begin{lem}\label{oracle.ineql}
If $||\theta_n||^2 \to \infty$ as $n \rightarrow \infty$, then we have
\begin{align*} 
[\,\text{a}\;]\;\;\;&\rho\big(\theta_n,g[X_n,1+r]\big)=2^{-1}\,n\, \log \left( 1+ r^{-1}\right).\\
[\,\text{b}\;]\;\;\; &\rho\big(\theta_n,g[\that^{\,JS}_n,r]\big)\sim (2 r)^{-1} \,n\, a_n\,(1+a_n)^{-1} \text{ where } a_n=n^{-1}\,||\theta_n||^2. \\ 
[\,\text{c}\;]\;\;\;&\rho\big(\theta_n,g[\widehat{\theta}_n^{\;JS},\,1+r]\,\big)\sim 2^{-1}\,n\, \left\{\, \log ( 1+ r^{-1}) - (1+a_n)^{-1}(1+r)^{-1} \right\}. \\ 
[\,\text{d}\;]\;\;\; &\rho\big(\theta_n,g[\widehat{\theta}_n^{\;JS}]\,\big)\sim 2^{-1}\,n\, \log \left \{ 1+ r^{-1} \, a_n\,(1+a_n)^{-1} \right\}.
\end{align*}
\end{lem}
The improvement in predictive risk due to efficient choice of location is reflected by the risk of $g[\widehat{\theta}^{\;JS},\,1+r]$ where as the effect of the optimal choice of scale after choosing an appropriate location estimate can be followed by evaluating the asymptotic predictive risk of $g[\widehat{\theta}^{\;JS}]$. 
\paragraph{}
The regularity conditions that we impose on the location point estimates do not extend to convex collections of estimates in $\mathcal{A}$.
But, the predictive risk still concentrates and the optimal predictive risk $\rho_0$ can be determined.
\begin{lem}
For any countable collection $\Lambda$ of  estimators $\that[\lambda]$ in $\mathcal{A}$ and their convex collection 
$\that^w=\sum_{\lambda \in \Lambda} w_{\lambda} \, \that[\lambda]$ with $\sum_{\lambda \in \Lambda} w_{\lambda}=1$, we have
\begin{align*} 
\rho_0(\theta_n,\that^{\,w}_n) - \frac{n}{2} \sum_{\lambda \in \Lambda} w_{\lambda}\,\log \left ( 1+ (nr)^{-1}\, \cdot \,
q\big(\theta_n,\that_n[\lambda]\big) \right) \,
 \leq \bigo\big(1\big) \; \text{ as } n \to \infty.
\end{align*}
And, the predictive density estimate $\sum_{\lambda \in \Lambda} w_{\lambda}\, g\big[\that_n[\lambda]\big]\big)$ is asymptotically optimal in the sense
\begin{align} 
\rho\,\bigg(\,\theta_n,\sum_{\lambda \in \Lambda} w_{\lambda}\, g\big[\that_n[\lambda]\big]\,\bigg)\, - \,\rho_0\big(\theta_n,\that^{\,w}_n\big) \,
 \leq \bigo\big(1\big) \; \text{ as } n \to \infty.
\end{align} 
\end{lem}
\paragraph{}
The class $\mathcal{G}[p]$ of product Gaussian density estimates is minimax optimal over ellipsoids \citep{Xu10}.  
However, we show that the class $\mathcal{G}[p]$ is minimax sub-optimal over the $\ell_0$--sparsity constrained space $\Theta(n,s)$ when $s/n \to 0$ as $n \to \infty$.
\begin{thm}\label{thm.minimax.risk}
For any fixed $r \in (0,\infty]$ as $n \to \infty$ for every sequence $s_n$ with $s_n/n \to 0$, we have
\begin{align}
\min_{\phat \in \mathcal{G}_n[p]} \;\;\;\max_{\Theta(n,s)}\; \rho(\theta_n,\phat)= r^{-1}\,s \log (n/s) (1+o(1)).
\end{align}
\end{thm}
By Theorem 1.2 in \citet{l0-sparse}, we know that the asymptotic minimax risk $R(n, s, r)$ over $\Theta(n,s)$ is given by
\begin{align}
R(n, s, r) \sim (1 + r)^{-1}\, s \,\log(n/s) \text{ as } n \to \infty, s\to \infty \text{ and } s/n \to 0.
\end{align}
Hence, the minimax sub-optimality of the class $\mathcal{G}[p]$ over the $\ell_0$ sparse space is $1+r^{-1}$.
The parametric space $\Theta(n,s)$ is not invariant to the group of orthogonal transformations. If the parameter space does not have any sparser representation with respect to the group of orthogonal transformations, then the asymptotic sub-optimality of the class $\mathcal{G}$ is also $1+r^{-1}$.

\subsection{Organization of the paper}
The predictive error of the class $G[1]$ of predictive density estimates is presented in the next section, which discusses the role of shrinkage in high-dimensional prediction problems. Predictive error and restricted minimax risk of the class $\mathcal{G}$ is presented in Section~3.

\section{Role of shrinkage and optimal error in $\mathcal{G}[1]$}\ssskip
Hereon we will assume that $\sigma_p^2=1$ and $\sigma_f^2=r$. 
The general predictive KL risk  will not be affected by this restriction. However, the density estimates are usually based
on statistics equivariant to the scale transformation and needs multiplication by $\sigma_p$.
\paragraph{Heuristic Idea:}
In the high dimensions  the quadratic loss of a reasonable point estimator will concentrate around its risk
And, so the KL risk of the corresponding Gaussian predictive density partitions into two parts involving 
(i) quadratic risk on the location parameter adjusted by the expected scale
(ii) logarithm of the expected scale.
As such, the risk $\rho\big(\theta_n,g[\widehat{\theta}_n,\widehat{c}_n]\big)$ of the normal predictive density estimate $g[\widehat{\theta}_n, \,\widehat{c}_{n}]$ is given by
$$\frac{n}{2} \left[ \left \{ \ex_{\,\theta_n}(\log \widehat{c}(X_n))+\ex_{\,\theta_n} \left(\frac{1}{\widehat{c}\,(X_n)}\right)-1 \right \}+\frac{1}{nr} \ex_{\,\theta_n} \left (\frac{\Vert \widehat{\theta}\,(X_n)-\theta_n \Vert^2}{\widehat{c}\,(X_n)} \right)
\right ].$$
In high dimensions, due to concentration of measure we expect 
\begin{itemize}
\item $\ex_{\theta_n} \log \big (\widehat{c}_n \big)\sim \log \ex_{\theta_n} \widehat{c}_n$\\[-1ex]
\item $\ex_{\theta_n} \big ( \widehat{c}_n^{-1}\big) \sim \big ( \ex_{\theta_n} \widehat{c}_n\big)^{-1}$\\[-1ex]
\item $\ex_{\,\theta_n} \left (\Vert \widehat{\theta}\,(X_n)-\theta_n \Vert^2 \,\cdot\, \widehat{c}^{-1}\,(X_n) \right)
\sim  \big(\ex_{\,\theta_n}\widehat{c}^{-1}\,(X_n)\big)^{-1}q(\theta_n,\that_n)$
\end{itemize}
which will lead to
$$\rho\big(\theta_n,g[\widehat{\theta}_n,\widehat{c}_n]\big)
\sim \frac{n}{2}   \left \{ \log \ex_{\theta_n}\widehat{c}_n 
+  \frac{1+  (nr)^{-1}q(\theta_n,\that_n)}{\ex_{\theta_n} \widehat{c}_n} -1 \right \} + \bigo(1)\; \text{ as } n \to \infty.
$$
This asymptotic decomposition of the predictive risk can be explicitly validated through the RASL properties. Because of this decomposition for any fixed point estimate $\widehat{\theta}_n$ at each parametric value $\theta_n$ we can minimize the above asymptotic value of $\rho(\theta_n,g[\widehat{\theta}_n, \,. \,])$ over the scalar quantity $\ex_{\theta_n}\widehat{c}_n$. 
The minimum asymptotic value is given by
$$\rho\big(\theta_n,g[\widehat{\theta}_n]\big) \sim n/2 \cdot \log \big ( 1+  (nr)^{-1} q(\theta_n,\that_n) \big)$$   
and the optimal value is attained when 
$$ \ex_{\theta_n}(\,\widehat{c}^{\, \text{opt}}\,(X_n)\,)={1+  (nr)^{-1}\,q(\theta_n,\that_n)}=\idealf_{\theta_n}(\that_{n})$$
which is the ideal flattening coefficient. However, $\idealf_{\theta_n}(\that_{n})$ is unknown. But, it depends only on the parametric value $\theta_n$.
Thus a choice would be $\widehat{c}^{\, \text{opt}}\,(X_n)=1+  (nr)^{-1}U[\widehat{\theta}_n](X_n)$, where  $U[\widehat{\theta}_n](X_n)$ (to be abbreviated as $\widehat{U}_n$) is reasonable (i.e with reasonable bias and concentration properties) estimate of the quadratic risk of $\widehat{\theta}_n$.
With very high probability such an optimal choice of $\widehat{c}_n$ will be greater than 1 reflecting a flattening of scale of the estimated predictive density (with respect to the true future variability). Intuitively,  we are performing an appropriate flattening of the density based on empirical estimates of the quadratic loss.  The optimal density $g[\widehat{\theta}_n]$ levels out with increasing inaccuracy in the location estimate $\widehat{\theta}_n$. 
\par 
One of the popular Frequentist notion (which is better than plug-in density estimates) of constructing predictive densities in this parametric model is to use Gaussian density estimate around an efficient location $\widehat{\theta}_n$ and variance $r+ \widehat{\var}(\widehat{\theta}_n)$.  Estimates of these kind are natural extensions of confidence sets. The optimal density estimate $g[\that]$ is quite similar except with a larger variance $r+\widehat{q}\,(\theta_n,\that_n)$. And unless the bias of $\theta_n$ is negligible compared to its variance the above mention general notion produces sub-optimal density estimates. Next through the RASL conditions we will quantify  some statistical regularities  in the behavior of quadratic loss in high dimensions. 
\subsection{RASL Properties of a Location Point Estimate} \label{rasl}
In high dimensions, for any fixed location parameter $\theta_n$ and its estimate $\widehat{\theta}_n$ we expect the quadratic loss $\Vert \, \widehat{\theta}_n-\theta_n\Vert_2^2$ to be concentrated around its expected value (quadratic risk) $ \ex_{\theta_n}\Vert \, \widehat{\theta}_n-\theta_n\Vert_2^2$ and it would be reflected by its variance. We will also rule out very bad point estimators by neglecting those with too high risk as we do not want them for prediction purposes. Apart from these we also assume the existence of a statistic which estimates the quadratic risk within reasonable bias. These properties of point estimators are referred to as \textbf{R}easonable \textbf{A}symptotic  \textbf{S}quare \textbf{L}oss properties and the corresponding location estimates as RASL estimates.

As  dimension $n \rightarrow \infty$, for any fixed  parametric value $\theta_n$ the location point estimate $\widehat{\theta}(X_n)$ is such that its quadratic loss has the following properties.
\begin{description}
\item [P1.] Reasonable Risk:
$$  \mathbf{\ex_{\theta_n}\Vert \,\widehat{\theta}_n-\theta_n\,\Vert^2 \leq O(n)} .$$
The canonical minimax point estimator  $X_n$  which is also the UMVUE (under square loss) in this case acts as the benchmark in weeding out the bad point estimators. 
For any parameter value $\theta_n$, $X_n$ has constant risk $n$. So, it is appropriate for our purpose to restrict ourselves to point estimators with risk of the O(n).\\
\item [P2.] Concentration property of Quadratic loss: 
$$\mathbf{ \var_{\theta_n\,}( \,\Vert \,\widehat{\theta}_n - \theta_n \,\Vert^2 \,) \leq O(n)}.$$
In high dimensions the estimator $\widehat{\theta}_n$ is such that its loss has variability less than $O(n)$. Again comparing with $X_n$, we see  
$\var_{\theta_n\,}(\,\Vert \,X_n - \theta_n \,\Vert^2\,)= 2 n$ as $\Vert \,X_n - \theta_n \,\Vert^2$ is distributed as a central $\chi^2$ random variable with $n$ degrees of freedom. 
\par
\textbf{P2} implies concentration of the loss function and would in turn also impose some concentration properties on well-behaved functions of the loss. As such, 
using Lemma \textbf{A.1},  we have
$$ \mathbf{P2.a \qquad \qquad \var_{\theta_n}\left\{\left(1+\frac{\Vert \widehat{\theta}_n - \theta_n \Vert^2}{nr}\right)^{-1}\right\}\leq O(n^{-1})} $$
following directly from \textbf{P2}. It is an important condition and  will be used in our derivations.\\
\item [P3.] Reasonable Estimate of Quadratic Risk: \vspace{0.5cm} \\
There exists an estimator $U[\,\widehat{\theta}_n\,](X_n)$ (will be abbreviated as $\widehat{U}_n$) of the quadratic risk of $\widehat{\theta}_n$ satisfying the following:
\begin{align*}
&\mathbf{P3.1. \qquad \qquad \left | \, \ex_{\theta_n} \big(\widehat{U}_n\big) -  \ex_{\theta_n} \, \Vert \widehat{\theta}_n - \theta_n\Vert^2 \, \right | \leq O(\,n^{1/2}\,) }.\\ 
&\\
& \mathbf{P3.2. \qquad \qquad \var_{\theta_n}\big(\widehat{U}_n\big) \leq O(n) }.\\
&\\
& \mathbf{P3.3. \qquad \qquad \var_{\theta_n}\left\{\left(1+\frac{\widehat{U}_n}{nr}\right)^{-1}\right\}\leq O(n^{-1})}.\text{\hspace{10cm}} 
\end{align*}  
\textbf{P3.1} implies existence of  a statistic which estimates the quadratic risk by not making significant bias.
Bias exceeding $O(\sqrt{n})$ is considered significant here and the order is associated with the $O(n^{-1})$ asymptotic statements we would like to make.
\textbf{P3.1} and \textbf{P3.2} are analogous to $\textbf{P2}$ and \textbf{P2.a} respectively. They imply that the asymptotic concentration properties associated with the quadratic loss also holds for its estimator $\widehat{U}_n$. If $\widehat{U}_n$ is positive then \textbf{P3.2} follows directly from \textbf{P3.1} by Lemma \textbf{A.1}.  
\end{description} 

\subsection{Validating  the RASL properties}
Given a location point estimator and its corresponding reasonable quadratic risk estimate the RASL conditions can be checked at least by simulations.
However,  existence of a `reasonable' risk estimator (as defined in \textbf{P3}) is essential.
For most  widely used point estimates, we can construct risk estimates satisfying the three conditions in \textbf{P3} though the procedures can sometime get quite complicated.
\par
If  $\widehat{\theta}_n$ is the posterior  mean -- generalized Bayes estimate with respect to prior $\pi$, then by Tweedie's formula \cite{Robbins54,Brown71} we have  explicit expression of an unbiased estimate of its risk as,
\begin{align*}
&\widehat{U}_{n}^{\,\pi}=n-\left[\,\Vert \nabla \log m_{\pi}(X_n)\Vert^2 - 2\frac{\nabla^2 m_{\pi}(X_n)}{m_{\pi}(X_n)} \, \right] \text{ where }
\nabla f \stackrel{\Delta}{=} \sum_{i=1}^n D_i \,f \\
& \nabla^2 m_{\pi}(X_n)=\sum_{i=1}^n D_i^2 \, m_{\pi}(X_n) \text{ and }
m_{\pi}(x_n)=\int \phi_n(x_n|\theta_n,1)\pi(\theta_n)\,d\theta_n.
\end{align*} 

$\widehat{U}_{n}^{\,\pi}$ is a natural candidate for a `reasonable estimate of the quadratic loss' though \textbf{P3.2} and \textbf{P3.3} are also to be checked separately. 
In particular, for \textbf{P3.3} to hold $\widehat{U}_{n}^{\,\pi}$ may need some modification
by introducing some bias. 

For spherically symmetric estimators, we can get candidates for `reasonable' risk estimates by using Stein's unbiased (quadratic) risk estimates (SURE) or their modifications (like positive part, etc) \cite{Stein74,Stein81}. As mentioned before, here too we needed to introduce some bias to the the SURE estimate as the unbiased one does not has property \text{P3.3}. 
\par
These RASL conditions are quite mild and usually holds for reasonable point estimates and can be checked by Monte Carlo simulations for arbitrary point estimates. 
Next, we check these conditions analytically for the following popular point estimators:
\begin{description}
\item [$\widehat{\theta}^{JS}$]: James Stein estimator\\[-2ex] 
\item [$\widehat{\theta}^{JS+}$]: Positive part James Stein estimator\\[-2ex]
\item [$\widehat{\theta}^{H}$]: Posterior mean of harmonic prior $\pi_H(\theta_n) \propto ||\theta_n ||^{-(n-2)}$.\\[-2ex]
\end{description}
All these $3$ point estimators are linear estimates of the form $s(X_n) X_n$ where $s(X_n)$ is a data-dependent shrinkage term. They are better than the canonical minimax estimator $X_n$. While $\widehat{\theta}^{H}$ is admissible, $\widehat{\theta}^{JS}$ and $\widehat{\theta}^{JS+}$ are both inadmissible. 
As such both $\widehat{\theta}^{JS+}$ and $\widehat{\theta}^{H}$ dominates $\widehat{\theta}^{JS}$. However, in high dimensions, they behave similarly and have near ideal linear risk properties. We will construct reasonable risk estimates for each of these estimators. 
While verifying the RASL conditions for the JS estimator we would also compute the  bound explicity for each $n$. It will be needed afterwards in Theorem~\ref{thm.dim.indep}. Since, the estimators are spherically symmetric it will be more informative to derive bounds depending on $\Vert \theta_n \Vert^2$. Hence forth in this section, by $a_n$ we denote $\Vert \theta_n \Vert^2/n$.   
A convenient fact about this spherically symmetric estimators is that the n-dimensional parameter $\theta_n$ can be substituted by $(||\theta_n||,0,\ldots,0)$ while checking the asymptotic behavior of square loss. As these estimators are not  Lipchitz functions of the normal random variable $X$, we can not directly use well-established Gaussian concentration inequalities \citep{Dembo-Zei-book,Ledoux-book}. 
\subsubsection{James Stein estimator}
The James-Stein estimator and its unbiased risk estimate is given by:
$$
\widehat{\theta}^{\,JS}_n= X_n\left(1- \frac{n-2}{\Vert X_n\Vert^2}\right)\,, \quad \text{and}\quad
 U (\,\widehat{\theta}^{\,JS}_n\,)= \left(n-\frac{(n-2)^2}{\Vert X_n \Vert^2}\right).
$$
RASL property \textbf{P1.} holds as the JS is better than the canonical estimator $X_n$. As such a good upper bound on its risk is also known
$$ 
E_{\theta_n}\Vert \,\widehat{\theta}^{\,JS}-\theta_n \,\Vert^2\leq 2+ \frac{(1-2/n) \, a_n}{(1-2/n) + a_n}.
$$
\begin{lem}
$$\var_{\theta_n}\bigg(||\widehat{\theta}_n^{JS}-\theta_n||^2\bigg) \leq 4\bigg [2n+ (n-2)^4 n^{-3} k_1(n) + n k_2(n) \bigg].$$
\end{lem}
\begin{proof}
We decompose $||\widehat{\theta}_n^{\,JS}-\theta_n||^2$ into 3 parts as 
$$||\widehat{\theta}_n^{\,JS}-\theta_n||^2= \Vert X_n-\theta_n\Vert^2+(n-2) ^2\Vert X_n \Vert^{-2}+ 2 (n-2)  M_n$$
where $M_n=\left\langle X_n-\theta_n,\,{X_n}\Vert X_n \Vert^{-2}\right\rangle$.
Then we use the naive inequality that for any three random variables $Z_i, i=1,2,3$ 
$$\var\bigg(\sum_{i=1}^3 Z_i\bigg)\leq \sum_{j=0}^3 \var \bigg(\sum_{i=1}^3 (-1)^{\I\{j=i\}} Z_i\bigg)=4 \sum_{i=1}^3 \var(Z_i)$$
to get the following bound on $\var_{\theta_n}\big(||\widehat{\theta}_n^{\,JS}-\theta_n||^2\big)$
\begin{align*}
\leq 4\bigg\{\var_{\theta_n}\big(\Vert X_n-\theta_n\Vert^2\big)+\var_{\theta_n}\bigg(\frac{n-2}{\Vert X_n \Vert^{2}}\bigg)
+4 \,(n-2)^4  \var_{\theta_n}\big( M_n\big)\bigg\}.
\end{align*}
Now $\Vert X_n-\theta_n\Vert^2$ has a central chi-square distribution with $n$ degrees of freedom and hence its variance is $2 n$. 
The bounds on the other quantities follow from Lemma~\ref{lem.k2} and Lemma~\ref{lem.k1}.
\end{proof}

\begin{lem}\label{lem.k2} 
For $n \geq 10$ we have $\var_{\theta_n}(M_n)\leq n^{-1} k_2(n)$ where 
\begin{align*}
k_2(n)=\max \{ h(n),k(n) \} \text{ where } & e_n = \frac{\sqrt{3}}{\prod_{i=1}^4 (1-(2i+1)/n)\}^{1/2}} \\
\text{and } & f_n=\frac{ 1}{(1- (\log n / n)^{1/2} )^{2}}.
\end{align*}
\end{lem}

\begin{proof}
The variance of  $M_n$ is same as the variance of $\langle \theta_n,X_n\rangle\Vert X_n \Vert^{-2}$ whose distribution is spherically symmetric in $\theta_n$ as it can be written as sum of two spherically symmetric terms $ H_n=\langle \theta_n, X_n - \theta_n \rangle ||X_n||^{-2}$ and $J_n= \Vert \theta_n\Vert^2 ||X_n||^{-2}$, . So, with out loss of generality we can assume that $\theta_n=(\theta,0,\ldots,0)$ where $\theta=\Vert \theta_n \Vert$.
We also divide the proof into two cases depending on the magnitude of $\theta$.
\par
When $\theta \leq \sqrt{n}$ we have,
$\var_{\theta_n}(M_n)\leq 2 (\, \var_{\theta_n}(H_n)+\var_{\theta_n}(J_n))$
with the later being less than $n^{-1}$ by Lemma~\ref{append-lem-1}.
And, the former is bounded above by $E(H_n^2)$. Now, with $Z \stackrel{d}{=} N(0,1)$ and $W\stackrel{d}{=}\chi^2_{n-1}(0)$ and $V=(Z+\theta)^2+W$ it can be rewritten as 
\begin{align*}
E\big(\theta^2 Z^2 V^{-2}\big) & \leq \sqrt{ \theta^4 E{Z^4} \, EV^{-4} }  \leq \sqrt{3 \theta^4 E(W^{\,-4}) } \leq \left \{ \frac{3 \theta^4}{\prod_{i=1}^{4} (n-2i-1)} \right \}^{1/2}
\end{align*}
which is less than $n^{-1} \sqrt{3} \prod_{i=1}^4 (1-(2i+1)/n)^{-1/2}$.
\par
When $\theta >  n$, we first recall that $M_n \stackrel{d}{=} (\theta+Z)/W$ and so
\begin{align*}
E(M_n^2)&\leq E\big\{V^{-2} \I_{\{|\theta+Z|\leq 1\}}\big\} + E\big \{ (\theta+Z)^{-2} \I_{\{|\theta+Z|> 1\}} \} \\
& \leq E\{ V^{-2}\} + 2 \int_1^{\infty} x^{-2} \phi(x-\theta) \, dx \\
& \leq [(n-3)(n-5)]^{-1} +
 \widetilde{\Phi}(\sqrt{\log n}) + \big\{ \theta - \sqrt{\log n} \big\}^{-2} \\
& \leq [(n-3)(n-5)]^{-1}+ n^{-1} (\log n)^{-1} + n^{-1} (1- (\log n / n)^{1/2} )^{-2}
\end{align*}
Hence the result follows.
\end{proof}

Though it is very tempting but we can not use the unbiased risk estimate $U (\,\widehat{\theta}^{\,JS}_n\,)$ as the estimate can be negative and violates \textbf{P3.3}.
\begin{lem} For any fixed $n$ and $r$,\,
$E_0 \big[\big \{ 1 + (nr)^{-1} U (\,\widehat{\theta}^{\,JS}_n\,) \big\}^{-1}\big]$ does not exist.
\end{lem}
We will instead use $\widehat{U}_n^+$ the positive part of $U (\,\widehat{\theta}^{\,JS}_n\,)$ and the scale estimate 
$\chat^{JS+}_n=1+(nr)^{-1}\widehat{U}_n^+$. RASL condition \textbf{P3.1} can be easily checked as
$$\var_{\theta_n}(\widehat{U}_n^{+})\leq \var_{\theta_n}(\,U (\,\widehat{\theta}^{\,JS}_n\,)\,)=(n-2)^4 \var_{\theta_n}(\,\Vert X_n \Vert^{-2})=O(n)$$
by Lemma~\ref{lem.k1} and \textbf{P3.2} follows from Lemma~\ref{append-lem-1}. As such, an exact dimension dependent bound can also be derived.

\begin{lem} For any fixed $n \geq 3 $ we have
$$ |Bias_{\theta_n}(\widehat{c}^{\,\text{JS+}}_n)| \leq k_3(n)\, n^{-1/2} \text{ where } k_3(n) = \frac{\sqrt{2}+5 n^{-\nicefrac{1}{2}} }{1-2/n}.$$
\end{lem}

\begin{proof}
Noting that $\bias_{\theta_n}(\widehat{c}^{\,\text{JS+}}_n) = n^{-1/2}\,\ex_{\theta_n}(\Uhat_n^-)$ and  
$$n \; \big|\ex_{\theta_n}( \Uhat_n^-) \big| \leq \ex_{\theta_n} \bigg[ \bigg(\frac{n}{Y}-1\bigg) \cdot \I\{Y \leq n \} \bigg]$$
where $Y$ follows Chi-square with degree $n$ and non-centrality parameter $||\theta_n||^2$.
We know that $Y \stackrel{d}{=} \chi^2_{n+2N}$ where $N \stackrel{d}{=} \text{Poisson}(||\theta_n||^2/2)$ and the above expectation can be written as 
$$\ex_{||\theta_n||^2} \bigg ( \ex \bigg[ \bigg(\frac{n}{Y_{n+2N}}-1\bigg) \cdot \I\{Y_{n+2N} \leq n \} \bigg \vert N \bigg] \big) \leq  \ex\bigg[ \bigg(\frac{n}{Y_n}-1\bigg) \cdot \I\{Y_n \leq n \} \bigg]$$
where $Y_{n+2N}$ is a central chi-square random variable with $(n+2N)$ degrees of freedom and the second inequality follows as for any $N\geq 0$,  $(n/Y_{n+2N}-1) \cdot \I\{Y_{n+2N} \leq n \}$ is stochastically dominated by $N=0$. Now,
\begin{align*}
\ex \bigg[ \bigg(\frac{n}{Y}-1\bigg) \cdot \I\{Y \leq n \}  N \bigg] =
\int_{0}^{n} \, \frac{n-y}{y} \;\frac{y^{n/2-1} e^{-y/2}}{2^{n/2} \Gamma(n/2)}\; dy \leq \frac{E|W_n-n|}{n-2} 
\end{align*}
where $W_n \sim \text{Gamma}(n/2-1,1/2)$ and so 
$$E|W_n-n|\leq 1+ E(W_n-n)^2 \leq 1+4+Var(W_n)=5+\sqrt{2n}$$
where the second inequality follows by Bias-Variance decomposition. Thus, we get our result. 
\end{proof}



%

\subsubsection{James-Stein Positive part Estimator}
We consider the positive part of the JS estimator and a reasonable estimate of its loss as
$$
\widehat{\theta}^{\,JS+}_n= X_n\left(1- \frac{n-2}{\Vert X_n\Vert^2}\right)_+ \quad \text{and}\quad
U \big(\,\widehat{\theta}^{\,JS+}_n\,\big)= \left(n-\frac{(n-2)^2}{\Vert X_n\Vert^2}\right)_+.
$$
There exists unbiased estimator of the quadratic risk of $\widehat{\theta}^{\,JS+}_n$ \cite[Exercise 2.13]{Johnstone-book}. We use a biased estimator here mainly to highlight the fact that even biased estimators will work. \\
\textbf{P1.} follows from the fact that $\widehat{\theta}^{\,JS+}_n$ is better than $\widehat{\theta}^{\,JS_n}$\cite[Exercise 2.8]{Johnstone-book}.
\smallskip\\
For checking \textbf{P2}, define $C_n$ to be the event $\{X_n:\widehat{\theta}^{\,JS+}(X_n) \neq 0\}=\{X_n:\widehat{\theta}^{\,JS+}_n =\widehat{\theta}^{\,JS}_n \}.$ 
And the idea is to relate the variance of the loss in JS+ case with the case of JS estimator.
\begin{align*}
&\displaystyle \var_{\theta_n}\big(\Vert \widehat{\theta}^{\,JS+}_n - \theta_n\Vert^2\big) 
=\displaystyle \ex\Vert \widehat{\theta}^{\,JS+}_n - \theta_n\Vert^4 - \ex^2\Vert \widehat{\theta}^{\,JS+}_n - \theta\Vert^2\\[1ex]
&\displaystyle =\ex\big\{\Vert \widehat{\theta}^{\,JS}_n - \theta_n\Vert^4 I_{C_n}\big\}- \ex^2\big\{\Vert \widehat{\theta}^{\,JS}_n - \theta_n\Vert^2 I_{C_n}\big\}+ ||\theta_n||^4 P(C_n^c) - ||\theta_n||^4 P^2(C_n^c) \\[1ex]
&\displaystyle=\var_{\theta_n}\big(\Vert \widehat{\theta}^{\,JS}_n - \theta_n\Vert^2 \big \vert \,C_n\big)\cdot P_{\theta_n}(C_n) +  ||\theta_n||^4P_{\theta_n}(C_n) P_{\theta_n}(C_n^c)\\[1ex]
&\displaystyle \leq \var_{\theta_n}\big(\Vert \widehat{\theta}^{\,JS}_n - \theta_n\Vert^2 \big)+ ||\theta_n||^4 P_{\theta_n}(C_n^c) \\[1ex]
&\displaystyle \;\;\text{as}\;\; \var_{\theta_n}\big(\Vert \widehat{\theta}^{\,JS}_n - \theta_n\Vert^2 \big) \geq \ex_{\theta_n} \left ( \var_{\theta_n}\big(\Vert \widehat{\theta}^{\,JS}_n - \theta_n\Vert^2 \big \vert C_n \big) \right)
\end{align*} 
We know that $\var_{\theta_n}\big(\Vert \widehat{\theta}^{\,JS}_n - \theta_n\Vert^2 \big)$ is $O(n)$ and lemma A.5 shows  $||\theta_n||^4 P_{\theta_n}(C_n^c) \leq O(n)$. So, we have the desired bound.
\paragraph{Condition \textbf{P.3.1}} We will condition on the event $C_n$ again and express \textbf{P.3.1} in terms of the James-Stein estimator
\begin{align*}
\ex_{\theta_n}U(\widehat{\theta}_n^{\,JS+}) - q(\theta_n,\widehat{\theta}_n^{\,JS+})
=\ex_{\theta_n}\bigg\{\bigg(U(\widehat{\theta}_n^{\,JS}) -  \Vert \widehat{\theta}_n^{\,JS}-\theta_n\Vert^2 \bigg)\,I_{C_n}\bigg\}-\Vert \theta_n \Vert^2 P\big(C_n^c\big).
\end{align*}
When $\theta_n=0$ then the R.H.S for large  $n$ reduces to,
$$ I_n=E\bigg\{\bigg (n-\frac{n^2}{Y_n}-\left(1-\frac{n}{Y_n}\right)^2Y_n  \bigg)\,I_{C_n}\bigg\} $$
where $Y_n$ is an central chi-squared random variable with $n$ degrees of freedom.
Now, we decompose $I_n$ into
\begin{align*}
&I_n=I_n^1+2 I_n^2 \text{ where }\\
&I_n^1=E\left\{(n-Y_n)\,I_{C_n}\right\}\text{ and } I_n^2=E\left\{(n-n^2/Y_n)\,I_{C_n}\right\}
\end{align*}
We standardize $Y_n$ as $Z_n = (Y_n-n)/\sqrt{2n}$. We can use concentration inequalities on $Z_n$ and have, [need to make rigorous]
\begin{align*}
I_n^1 &\leq \sqrt{2n} \cdot EZ_+ \text{ as } n \rightarrow \infty \\
I_n^2&= \sqrt{n}\cdot E \left\{\left(\frac{Z_n}{Z_n/\sqrt{n}+1}\right) \,I_{C_n}\right\} \leq \sqrt{n}\, E|Z_n| \rightarrow \sqrt{n}E|Z|
\end{align*}
as on $C_n$, $Z_n \geq 0$. Thus $I_n \leq O(\sqrt{n})$.
\paragraph{Condition \textbf{P3.2}} 
By Lemma A.8 we have
$$\var_{\theta_n}(\,U (\,\widehat{\theta}^{\,JS+}_n\,)\,)\leq 
\var_{\theta_n}(\,U (\,\widehat{\theta}^{\,JS}_n\,)\,)=O(n).$$
\paragraph{Condition \textbf{P.3.3}} Follows from Lemma A.1
\\[1ex]
\textbf{Harmonic Prior}
The conditions can be checked for $\widehat{\theta}^{\,H}$ by using its closed form expressions in \cite[Chapter 2]{Xu-thesis07}.

\subsection{Determining $\rho_0$ for RASL point estimators}
In this section, we will show that in high dimension with very high precision we can express $\rho_0(\that_n)$ -- the minimum  Predictive Entropy risk of the class of Gaussian density estimates around location $\that_n$ in terms of the Mean Square estimation error of $\theta_n$ by $\that_n$. We initially prove bounds on the error rates which holds for all dimensions but are dimension dependent. Then, we would show that in high dimensions those bounds are asymptotically sharp.
\subsubsection{Lower Bound on $\rho_0(\theta_n,\,\widehat{\theta\,}_n)$: } 
Next, we produce a lower bound on the prediction error. The bound ultimately will be a function of $\theta_n$  though it depends on the form of $\that_n$. It involves expectation of a quantity which usually is neither a parameter nor a statistic and hence can not be computed in closed form.
\begin{lem}\label{lem-lowbound}
For any dimension $n$, any parameter value $\theta_n$ and any location point estimate $\widehat{\theta}\,(X_n)$,  we have
$$\rho_0(\theta_n,\widehat{\theta}_n)\geq \frac{1}{2}\, \ex_{\theta_n} \left \{ \log \left (1+ r^{-1} \cdot n^{-1}\,\cdot \, {\Vert \, \widehat{\theta}_n-\theta_n\Vert^2} \,\right) \, \right \}.$$
\end{lem}
\begin{proof}
For any fixed $n$, the risk of the predictive density $q_n$ 
which is a $n$-dimensional product normal with 
with data adaptive mean $\widehat{\theta}\,(X_n)$ and 
data and parameter dependent  variance $\widehat{c} \,(\theta_n,X_n) \, r$  (for all co-ordinates) is given by $ 2 \rho(\theta_n,q_n)$
\begin{align*}
&= n \left \{ \ex_{\theta_n}(\log \widehat{c}\,(\theta_n, X_n))+\ex_{\theta_n} \left(\frac{1}{\widehat{c}\,(\theta_n, X_n)}\right)-1 \right\}+\frac{1}{r} \ex_{\theta_n} \left (\frac{||\widehat{\theta}\,(X_n)-\theta_n||^2}{\widehat{c}\,(\theta_n, X_n)} \right)\\
&=n\, \ex_{\theta_n} \left \{ \log \widehat{c}\,(\theta_n, X_n)+ \frac{1+ (nr)^{-1}||\,\widehat{\theta}\,(X_n)-\theta_n||^2}{\widehat{c}\,(\theta_n, X_n)}-1 \right\}.
\end{align*}
For any  fixed value of $\theta_n$ and for each $x_n$,\\
$$\log \widehat{c}\,(\theta_n, x_n)+ \widehat{c}^{\,-1}(\theta_n, x_n) \{1+ (nr)^{-1}||\,\widehat{\theta}(x_n) -\theta_n||^2\}-1$$  
is minimized at  
$\widehat{c}^{\;\text{opt}\,}(\theta_n, x_n)= 1+ (nr)^{-1}||\widehat{\theta}(x_n)-\theta_n||^2 $ and the minimum value is given by \mbox{$ \log  (1+ (nr)^{-1}  \, \Vert \, \widehat{\theta}(x_n)-\theta_n\Vert^2 \,) $}.
Hence, the result follows.
\end{proof}
Though $\widehat{c}^{\;\text{opt}\,}(\theta_n, X_n)$ is the best possible flattening coefficient, it depends on the parameter and can not be used in practice. As such, 
$\widehat{c}^{\;\text{opt}\,}(\theta_n, X_n)$ is the ideal flattening coefficient.
In high dimensions due to statistical regularity we expect $\widehat{c}^{\;\text{opt}\,}(\theta_n, X_n)$ to be very close to its expected value 
$$\ex_{\theta_n}\{\widehat{c}^{\;\text{opt}\,}(\theta_n, X_n)\}=1+ (nr)^{-1} \ex_{\theta_n} \Vert \that_n -\theta_n \Vert^2$$ 
which can be viewed as the (near) \textbf{I}deal \textbf{F}lattening coefficient and is referred to as 
$\idealf_{\theta_n}(\widehat{\theta}_n)= 1+n^{-1}r^{-1}q(\theta_n,\widehat{\theta}_n)$. Here flattening coefficients are usually called scale and it should be noted that the corresponding variance needs to be multiplied by $r$. 
\paragraph{}
From Lemma \ref{lem-lowbound} we can derive a worse but more tractable bound
\begin{align}\label{lower-bound-P1}
\rho_0(\theta_n,\widehat{\theta}_n)\geq 2^{-1} \, \ex_{\theta_n} \left \{ \log \left ({\Vert \, \widehat{\theta}_n-\theta_n\Vert^2/(nr)} \,\right)\right\}. 
\end{align}
\subsubsection{Upper Bound for  $\rho_0(\theta_n,\that_n)$:}\label{lem-upbound}
We now produce an upper bound on the risk of any Gaussian density estimate.
Henceforth, $\std$ would mean Standard Deviation and by $\bias$ of the scale estimate $\widehat{c}_n$ we would mean the expected deviation from the near ideal flattening coefficient $\idealf_{\theta_n}(\that_n)$. With scales estimators based on the statistic $U[\widehat{\theta}_n](X_n)$ and of the form $\widehat{c}\,(X_n) = (1+n^{-1} U[\widehat{\theta}_n](X_n))$  we have
$\bias_{\theta_n} (c_n)=(nr)^{-1} \big [\ex_{\theta_n} U[\widehat{\theta}_n](X_n) - q(\theta_n,\that_n) \big ]$.

\begin{lem}\label{lem-upper-bound}
For any  fixed dimension $n$, parameter value  $\theta_n$, location point estimate $\widehat{\theta}(X_n)$ and any scale estimate $ \widehat{c}\,(X_n) > 0 $ almost surely and of the form $\widehat{c}\,(X_n) = 1+(nr)^{-1} U[\widehat{\theta}_n](X_n)$, we have
\begin{align*}
&\rho\big(\theta_n,g[\widehat{\theta}_n,\,\widehat{c}_n]\big) -
\frac{n}{2} \cdot \log \big (\idealf_{\theta_n}(\widehat{\theta}_n) \big)
 \leq \frac{n}{2} \cdot \bigg[ A_{\theta_n}\big(\that_n,\chat_n\big) + B_{\theta_n}\big(\that_n,\chat_n\big) \bigg] \medskip\\ 
&\text{where }  A_{\theta_n}\big(\that_n,\chat_n\big)  =  \idealf_{\theta_n}\big(\widehat{\theta}_n\big) \, \big \{ \ex_{\theta_n} (\widehat{c}_n)\big \}^{-1}\std_{\theta_n} (\widehat{c}_n) \std_{\theta_n}(\widehat{c}_n^{\;-1})\\
&\text{ {\color{white} where }} + r^{-1} SD_{\theta_n} \bigg( \frac{\Vert \widehat{\theta}_n - \theta_n \Vert^2}{n}\bigg) \, \std_{\theta_n}\big(\widehat{c}_n^{\;-1}\big) \;\; \text{and}\\
&\text{{\color{white} where }} B_{\theta_n}\big(\that_n,\chat_n\big)= \bias^{\,2}_{\theta_n} \big(\chat_n\big)\, \big\{ \idealf_{\theta_n}(\widehat{\theta}_n) \big\}^{-1} \, \big\{\ex_{\theta_n} (\widehat{c}_n)\big\}^{-1}.
\end{align*}
\end{lem}

\begin{proof}
The risk of the normal predictive density estimate $g[\widehat{\theta}_n\,, \,\widehat{c}_{\,n} \,]$  is given by $2 \rho(\theta_n,\widehat{g}_n)$
\begin{equation*}
= n \left \{ \ex_{\theta_n}(\log \widehat{c}(X_n))+\ex_{\theta_n} \left(\frac{1}{\widehat{c}(X_n)}\right)-1 \right\}+ \ex_{\theta_n} \left (\frac{||\,\widehat{\theta} \,(X_n)-\theta_n||^2}{r \,\widehat{c}(X_n)} \right).
\end{equation*}
Now, we  replace $\ex_{\theta_n} (\widehat{c}_n^{\;-1})  $ by $\ex^{-1}_{\theta_n} \widehat{c}_n\,$ and 
$ \ex_{\theta_n}\big (\, ||\widehat{\theta}_n-\theta_n||^2 \times  \widehat{c}^{\;-1}_n\,\big)$
 by $ {\ex_{\theta_n}||\widehat{\theta}_n-\theta_n||^2}\times \ex^{-1}_{\theta_n} \widehat{c}_n$ in the above expression to get $\tilde{\rho}(\theta_n,\widehat{g}_n)$
\begin{equation}\label{risk-decomp}
\begin{aligned}
&= \frac{1}{2} \left[ n \left \{  \ex_{\theta_n} ( \,\log \widehat{c}\,(X_n)\, )+ \frac{1}{\ex_{\theta_n}(\,\widehat{c}\,(X_n)\,)}-1 \right\}+\frac{1}{r} \left ( \frac{\ex_{\theta_n}||\widehat{\theta}\,(X_n)-\theta_n||^2}{\ex_{\theta_n}(\,\widehat{c}\,(X_n)\,)} \right)\right ]\\
&=\frac{n}{2}  \ex_{\theta_n} ( \, \log \widehat{c}\,(X_n)\, ) 
+ \frac{n}{2} \left \{ \frac{1+  (nr)^{-1}\ex_{\theta_n}\Vert\widehat{\theta}\,(X_n)-\theta_n\Vert^2}{\ex_{\theta_n}(\,\widehat{c}\,(X_n)\,)} -1 \right \}\\
&=\frac{n}{2}  \ex_{\theta_n} ( \, \log \widehat{c}\,(X_n)\, ) 
- \frac{n}{2}  \frac{\bias_{\theta_n} \big( \widehat{c}\,(X_n)\big)}{\ex_{\theta_n}(\,\widehat{c}\,(X_n)\,)}
\end{aligned}
\end{equation}
and the distortion caused thereby $(n/2)^{-1}(\rho(\theta_n,\widehat{g}_n)-\tilde{\rho}(\theta_n,\widehat{g}_n))$ equals
\begin{equation}\label{asym-cond}
\ex_{\theta_n} \left (\frac{1+(nr)^{-1}||\widehat{\theta}\,(X_n)-\theta_n||^2}{\widehat{c}\,(X_n)} \right) -  \frac{1+(nr)^{-1}\ex_{\theta}||\widehat{\theta}\,(X_n)-\theta_n||^2}{\ex_{\theta_n}(\,\widehat{c}\,(X_n)\,)}. 
\end{equation}
Next we will show that $(n/2)^{-1}\vert r(\theta_n,\widehat{g}_n)-\tilde{r}(\theta_n,\widehat{g}_n) \vert \leq A_{\theta_n}(\that_n,\chat_n)$. Before that, note that
if $\widehat{U}_n$ is unbiased then the second term in Equation~\ref{risk-decomp} vanishes and we have the result stated in Corollary~\ref{cor-unbiased-isometry}.\\   
Now, note that $2 n^{-1}\,\tilde{r}(\theta_n,\widehat{g}_n)$ equals
$$\log \idealf_{\theta_n}(\widehat{\theta}_n) 
+ \ex_{\theta_n}\left[ \log\left(1+\frac{\widehat{c}\,(X_n)- \idealf_{\theta_n}(\widehat{\theta}_n)}{\idealf_{\theta_n}(\widehat{\theta}_n)}\right) \right] - \frac{\bias_{\theta_n} \big( \widehat{c}\,(X_n)\big)}{\ex_{\theta_n}(\,\widehat{c}\,(X_n)\,)}$$
and using the inequality $\log(1+x)\leq x$ for all $x > -1$ on the second term on the right hand side it follows that
\begin{align*}
2 n^{-1}\,\tilde{r}(\theta_n,\widehat{g}_n) &\leq \log \idealf_{\theta_n}(\widehat{\theta}_n) 
+ \frac{\bias_{\theta_n} \big( \widehat{c}\,(X_n)\big)}{\idealf_{\theta_n}(\widehat{\theta}_n)} - \frac{\bias_{\theta_n} \big( \widehat{c}\,(X_n)\big)}{\ex_{\theta_n}(\,\widehat{c}\,(X_n)\,)} \\
&= \log \idealf_{\theta_n}(\widehat{\theta}_n) 
+ \frac{\bias^2_{\theta_n} \big( \widehat{c}\,(X_n)\big)}{\idealf_{\theta_n}(\widehat{\theta}_n) \ex_{\theta_n}(\,\widehat{c}\,(X_n)\,)}=B_n.
\end{align*}
Now,  we write $2 n^{-1} \{r(\theta_n,\widehat{g}_n)-\tilde{r}(\theta_n,\widehat{g}_n)\} = H_{\theta_n}(\that_n,\chat_n) + J_{\theta_n}(\that_n,\chat_n)$ where,
\begin{align*}
H_{\theta_n}(\that_n,\chat_n) & = \idealf_{\theta_n}\big(\widehat{\theta}_n\big) \left \{ \ex_{\theta_n}\left ( \frac{1}{\widehat{c}_n}\right) - \frac{1}{\ex_{\theta}\widehat{c}_n} \right\} \quad \text{ and }\\
 J_{\theta_n}(\that_n,\chat_n) & = (nr)^{-1} \cdot  \ex_{\theta_n} \left [ ||\widehat{\theta}_n-\theta_n||^2 \biggl \{ \frac{1}{\widehat{c}_n} -  \ex_{\theta_n} \left (\frac{1}{\widehat{c}_n} \right) \biggr \} \right ].
\end{align*}
Note that the second term in $H_{\theta_n}(\that_n,\chat_n)$ can be rewritten as, 
\begin{align*}
\ex_{\theta_n}\left ( \frac{1}{\widehat{c}_n}\right) - \frac{1}{\ex_{\theta_n}\widehat{c}_n}
=  \frac{-1}{\ex_{\theta} \widehat{c}_n} \cdot 
\ex_{\theta_n}\left [\bigg(  \widehat{c}_n - \ex_{\theta_n}\widehat{c}_n\bigg) \left (\frac{1}{\widehat{c}_n} - \ex_{\theta_n}\left(\frac{1}{\widehat{c}_n}\right)\right) \right ]
\end{align*}
which by Cauchy-Schwartz (C-S) inequality has lower absolute value than
$$ (\ex_{\theta_n} \widehat{c}_n )^{-1}\bigg\{\var_{\theta_n} (\widehat{c}_n) \times \var_{\theta_n}\big(\widehat{c}_n ^{\;-1}\big) \bigg\}^{1/2}.$$
Thus, $|H_{\theta_n}(\that_n,\chat_n)|\leq \idealf_{\theta_n}\big(\widehat{\theta}_n\big) \ex_{\theta_n} \widehat{c}_n )^{-1}
\std_{\theta_n} (\widehat{c}_n) \times \std_{\theta_n}\big(\widehat{c}_n ^{\;-1}\big)
$.\\
Again,  rewriting $J_{\theta_n}(\that_n,\chat_n)$ as,
$$ J_{\theta_n}(\that_n,\chat_n)=(nr)^{-1} \ex_{\theta_n} \left [ \left \{ ||\widehat{\theta}_n-\theta_n||^2-\ex_{\theta_n}||\widehat{\theta}_n-\theta_n||^2 \right \}\left \{ \widehat{c}_n^{\;-1} -  \ex_{\theta_n} \big(\widehat{c}_n^{\;-1}\big) \right \} \right ]$$
and applying C-S inequality we get $$|J_{\theta_n}(\that_n,\chat_n)|\leq
(nr)^{-1}\std_{\theta_n} \, \left (\Vert \widehat{\theta}_n-\theta_n \Vert^2 \right) \cdot \std_{\theta}\left(\widehat{c}_n^{\;-1}\right).$$
So $(n/2)^{-1\,}\vert r(\theta_n,\widehat{g}_n)-\tilde{r}(\theta_n,\widehat{g}_n) \vert \leq \vert H_{\theta_n}(\that_n,\chat_n) \vert + \vert J_{\theta_n}(\that_n,\chat_n)\vert \leq A_{\theta_n}(\that_n,\chat_n)$ and we have our desired result.
\end{proof}

\begin{cor}\label{cor-unbiased-isometry}
If  $\widehat{U}_n$ is an unbiased estimate of the parameter $q(\theta_n,\widehat{\theta}_n)$ and $\widehat{c}_n=1+(nr)^{-1} \widehat{U}_n > 0$ almost surely, then we have,
$$ \rho_0(\theta_n,\widehat{\theta}_n) \leq \rho(\theta_n,g[\widehat{\theta}_n\,,\,\widehat{c}_{n}]) \leq \frac{1}{2}\,\ex_{\theta_n}\left \{ \log \left ( 1+ (nr)^{-1} 
{\widehat{U}_n} \right)  \right\}  + A_{\theta_n}(\that_n,\chat_n)/2 .$$
\end{cor}
The corollary follows from the above Lemma. 
The upper bound derived here involves expectation of a statistic along with a distortion term $A_{\theta_n}(\that_n,\chat_n)$ which will be negligible under the RASL conditions. Ignoring it for the time being we can say that an upper bound is produced when $||\that_n-\theta_n||^2$ in the lower bound of Lemma~\ref{lem-upper-bound}  can be replaced by a good statistic.
Lemma~\ref{lem-upper-bound} has an upper bound based on $\idealf_{\theta_n}(\that_n)$ and next we show that the lower bound and the upper bound are fairly close.

\begin{lem}\label{lem-upperbound-2}
For any point estimate $\widehat{\theta}_n$ and location parameter $\theta_n \in \RR^n$ we have, 
\begin{align*}
&\rho_0(\theta_n,\,\widehat{\theta}_n)) \geq 2^{-1} \log \idealf_{\theta_n}\big( \widehat{\theta}_n\big) - L_{\theta_n}(\that_n)/2 \text{ where,} \\ 
&L_{\theta_n}(\that_n)= \big(nr\big)^{-1} \cdot SD_{\theta_n}\big(\Vert \widehat{\theta}_n - \theta_n \Vert^2\big) \cdot SD_{\theta_n} \bigg \{ \bigg( 1+ (nr)^{-1} \,\Vert \widehat{\theta}_n - \theta_n \Vert^2\bigg)^{-1}\bigg\} 
\end{align*}
\end{lem}

\begin{proof}
From Lemma~\ref{lem-lowbound}  we have
\begin{align*}
\log \idealf_{\theta_n}\big( \widehat{\theta}_n\big)- 2 \rho_0(\theta_n,\that_n) & \leq  \log \idealf_{\theta_n}\big( \widehat{\theta}_n\big) - \ex_{\theta_n} \big \{ \log  \big(1+ (nr)^{-1}{\Vert \, \widehat{\theta}_n-\theta_n\Vert^2} \,\big) \big\} \\
&=\ex_{\theta_n}\left \{ \log \left( 1 - \frac{\bar{l}(\theta_n,\that_n)}{nr+\Vert \,\widehat{\theta}_n- \theta_n \, \Vert^2} \right)\right \}
\end{align*}
where $\bar{l}(\theta_n,\that_n)=\Vert \,\widehat{\theta}_n- \theta_n \, \Vert^2 - q(\theta_n,\that_n) $and  using Jensen's inequality and  $\log(1+x)\leq x$ consecutively, the difference becomes
\begin{align*}
&\leq - \ex_{\theta_n}\left ( \frac{\bar{l}(\theta_n,\that_n)}{nr+\Vert \,\widehat{\theta}_n- \theta_n \, \Vert^2} \right )\\
&=- \ex_{\theta_n}\left [ \bar{l}(\theta_n,\that_n)\cdot 
\left \{
\frac{1}{nr+\Vert \,\widehat{\theta}_n- \theta_n \, \Vert^2} -
\ex_{\theta_n} \left(\frac{1}{nr+\Vert \,\widehat{\theta}_n- \theta_n \, \Vert^2}\right) \right \} \right ]\\
\intertext{and by applying C-S inequality the magnitude of the said difference is}
&\leq \std_{\theta_n}( \Vert \,\widehat{\theta}_n- \theta_n \, \Vert^2) \times  \std_{\theta_n}\big\{ (nr+\Vert \,\widehat{\theta}_n- \theta_n \, \Vert^2)^{-1}\big\}=L_{\theta_n}(\widehat{\theta}_n).
\end{align*}
This completes the proof.
\end{proof}

\begin{cor} Under the conditions of Lemma~\ref{lem-upper-bound} we have
\begin{align*}
[i.]\; & 0 \leq \rho\big(\theta_n,g[\widehat{\theta}_n,\widehat{c}_n\,]\big) - \rho_0(\theta_n,\that_n)\leq 2^{-1}\,\big\{ L_{\theta_n}(\widehat{\theta}_n) +[A+B]_{\theta_n}(\that_n,\chat_n) \big\}\\[1ex]
[ii.]\; & \big \vert \rho_0(\theta_n,\that_n) - 2^{-1} \log \idealf_{\theta_n}(\that_n) \big \vert \leq 2^{-1}\max\,\big\{L_{\theta_n}(\widehat{\theta}_n), [A+B]_{\theta_n}(\that_n,\chat_n)\big\}. 
\end{align*}
\end{cor}

The corollary follows directly by combining the above lemma with Lemma~\ref{lem-upper-bound}.  It bounds the deviation of the predictive risk from a continuous, increasing function of the $\mse$. The RASL conditions ensure the existence of at least one candidate for the statistic $\widehat{U}_n$ such that $c(X_n) >0$ almost surely (follows from RASL condition P3.3) and each of the associated terms $A_{\theta_n}(\that_n,\chat_n), B_{\theta_n}(\that_n,\chat_n)$ and $L_{\theta_n}(\widehat{\theta}_n)$ is of the order of $O(n^{-1})$. Hence, Theorem~\ref{rasl.main.thm} follows. 

\begin{proof}[Proof of Theorem~\ref{rasl.main.thm}]
Note that under the RASL conditions we have  $A_{\theta_n}(\that_n,\chat_n), B_n$ and $L_{\theta_n}(\widehat{\theta}_n)$ to be of the order of $O(n^{-1})$. Also, note that the fact that $c>0$ almost surely is taken care in the the RASL property P3.3.
\end{proof}

\subsection{Violation of the RASL conditions}
Based on the lower bound of  \ref{lem-lowbound} and concentrating around the expectation by using Chebyshev's inequality we have for any $a$ in $(0,1)$
$$\rho_0(\theta_n,\that_n) \geq \frac{1}{2} \, \log \bigg(1+ \frac{(1-a)\,q(\theta_n,\that_n)}{nr}\bigg) 
\bigg\{ 1- \frac{a^2\,\var_{\theta_n}(||\that_n -\theta_n ||^2)}{ q^2(\theta_n,\that_n)} \bigg\}.$$

So if $\textbf{P1}$ of the RASL condition is violated i.e   for some $\theta'_n \in \RR^n$ we have $q(\theta'_n,\that_n)> O(n)$ then if $\textbf{P2}$ holds or we have $\var_{\theta'_n}(||\that_n -\theta'_n ||^2)< q(\theta'_n,\that_n)$ then $\rho_0(\theta'_n,\that_n) > 1/2 \log(1+r^{-1}) $ which is the minimax risk of the best invariant density estimate and so the class of density estimates in $\mathcal{G}[1]$ centered around $\that$ does not have any minimax estimator. Thus, we can exclude bad point estimators in most conditions (also see Equation~\ref{lower-bound-P1}).
\par
Among the cases where RASL conditions does not hold the only exciting case is when \text{P2} is violated but \text{P1} holds. 
In those cases the asymptotic predictive entropy risk can not be characterized in closed form.
A example of a point estimator of this kind is:
\begin{align*}
\delta_n(i)= \left \{\begin{array}{cl} \delta_1(X_1) & \text{ if } i =1 \\ X_i & \text{ if } i =2,\cdots,n\end{array}\right.
\end{align*}
where the univariate point estimator $\delta_1$ is given by
\begin{align*}
\delta_1(x)=\left\{\begin{array}{cl} n^{1/2} \,(2\log n)^{-1/2} \, x & \text{ if } x < (2 \log n)^{1/2} \\ x & \text{ if } x \geq (2 \log n)^{1/2}\end{array}\right.
\end{align*}

\subsection{Decision Theoretic implications}
The asymptotic relation between the predictive risk and the mean square risk will help us in deriving oracle inequalities on the predictive risk of $g_n$.
The bounds will be sharp enough to discuss asymptotic optimality in the class ${\mathcal{G}}$.
We would first relate the class $\mathcal{G}$ with the other decision-theoretic classes of predictive densities.
Then, we would compare the predictive risk of the respective classes in unrestricted parametric spaces.
\par 
In the above context, we consider the following $6$  predictive estimates: 
\begin{itemize}
\item $\mathbf{\widehat{p}_L}$\;: As an representative of the class of all Linear predictive density estimates $(\mathcal{L})$ we choose the predictive density ${g[X_n,\,1+r]}$. It is the Bayes predictive density with respect to the uniform prior, has constant risk and is inadmissible in $\mathcal{L}$. It is the best invariant predictive strategy and is also minimax among all procedures\cite{Liang04}.\\ 
\item $\mathbf{\widehat{p}_\ex}$\;: We choose the James-Stein positive part  plug-in predictive density estimate ${g[\widehat{\theta}^{\;JS+}_n,\,r]}$ as a representative of $\mathcal{P}$. 
Though the positive part James-Stein estimator is inadmissible as a point estimate, it is difficult to find estimators that have significant improvements over it. 
And, for all practical purposes the JS+ estimator can be considered as a `nearly' admissible point estimate. 
In that respect we can consider 
$$\mathbf{\widehat{p}_\ex}=\mathbf{g[\widehat{\theta}^{\;JS+},\,r]} \;\;\text{where}\;\; \widehat{\theta}^{\;JS+}_n= X_n \, \left(1-\frac{(n-2)}{\Vert X_n\Vert^2}\right)_+$$
as an efficient representative from the class of Plug-in predictive densities $(\mathcal{P})$. The subscript stands for the class of estimative (plug-in) distributions.\\
\item $\mathbf{\widehat{p}_H}$\;: We consider  the Bayes predictive density estimate from the harmonic prior $\pi_H$ as a representative of the class of all Bayes predictive density estimates $(\mathcal{B})$. 
It is an admissible rule. As such, it also dominates $\widehat{p}_L$\cite{Komaki01,Ghosh08}.
\item Next, we consider 3 member of $\mathcal{G}$ which we will use to compare the risk of the predictive densities from the above 3 classes.
\begin{itemize}
\item $\mathbf{g[\widehat{\theta}^{\;JS+},\,1+r]}$: A non-linear, fixed variance predictive density estimator around the JS+ estimator. It is uniformly better than $\widehat{p}_L$. It is also denoted by $g_M$.
\item $\mathbf{g[\widehat{\theta}^{\;JS+}]}$:  The optimal member in $\mathcal{G}(\widehat{\theta}^{\;JS+})$ which we will use to compare with $\widehat{p}_\ex$ and $\widehat{p}_L$.
\item $\mathbf{g[\widehat{\theta}^{\;H}]}$:  The optimal member in $\mathcal{G}(\widehat{\theta}^{\;H})$. We would like to compare its performance with $\widehat{p}_H$. Also, $g[\widehat{\theta}^{\;H}]$ is asymptotically inadmissible among the procedures in ${\mathcal{G}}$.
\end{itemize}
\end{itemize}
In Table~\ref{table-baseball} we evaluate the predictive performance of each of these density estimates on a dataset. 
\subsubsection*{Oracle inequalities and Implications}\label{oracle-ineql}
Lemma~\ref{oracle.ineql} describes the predictive risk of density estimates center around $\that^{\,JS}$.
\begin{proof}[Proof of Lemma~\ref{oracle.ineql}]
The results follows from Theorem \ref{rasl.main.thm} and by using Proposition 2.6 and Exercise 2.8 of \cite{Johnstone-book}
\end{proof}
The lemma will not be useful in very very low signal-to-noise ratio. It can be used effectively when $a_n>O(n^{-1})$. Note that, we can partition the  improvement in the asymptotic prediction error over $\widehat{p}_{L}$ in two parts.
\begin{itemize}
\item We first shrink the location estimate while keeping the scale unperturbed and  move to a better estimate $g[\widehat{\theta}^{\;JS+},\,1+r]$. Let  the improvement be denoted by $d^1_n$.  
\item Now we optimize the scale keeping the location fixed and arrive at $g[\widehat{\theta}^{\;JS+}]$. Let the improvement be denoted by $d^2_n$.
\end{itemize}
And, based on the lemma we have, 
\begin{align*}
d^1_n\sim\frac{1}{2}\,\alpha_n \;\text{and}\; d^2_n\sim\frac{1}{2}\,\log(1-\alpha_n)^{-1} \;\; \text{where} \quad \alpha_n=\{(1+a_n)(1+r)\}^{-1}.
\end{align*}
As $\alpha_n < 1$, $d^1_n$, $d^2_n$ as well as $d_n^2-d_n^1$ are all positive and increasing in $\alpha_n$. 
It means we are actually making more improvement by adapting the scale than that we got by shifting  location and their difference is also decreasing in both $a_n$ and $r$. 

\subsubsection*{Prediction error for shrinkage estimators}
By shrinkage point estimators we define estimators of the form $s(X_n) \, X_n$ where $s(X_n)$ is an almost everywhere differentiable function. 
If $||\theta_n||^2$ were known, then spherically symmetric shrinkage estimators of the form $s(a_n)\,X_n$ where 
$a_n=||\theta_n||^2/n$ and $s(a_n) \leq 1$ would be efficient.
Let $\mathcal{S}$ denotes the class of normal predictive densities based on ideal point location estimators.
Such an estimate satisfies the RASL condition P2 and so Lemma~\ref{lem-upperbound-2} can be used to calculate an optimal lower bound on the predictive risk of the family of density estimators based on $S$ -- the class of all shrinkage point estimators conditioned on $a_n$. 
\par
Note that by Bias-Variance decomposition the quadratic risk of the ideal point estimator $s(a_n) \, X_n$ is given by
\begin{align*}
\ex_{\theta_n}\big(\Vert s_n\,X_n - \theta_n\Vert^2\big) &=s_n^2 \, n + \bar{s}^2_n ||\theta_n||^2 \text{ where } \bar{s}_n=1-s_n \text{ and } s_n=s(a_n).  
\end{align*}
Based on Lemma~\ref{append-lem-1} we have $L_{\theta_n}(\that_n)\leq (n r)^{-2} \var_{\theta_n}(\Vert \that_n - \theta_n\Vert^2)$ and for an estimator in ${S}$ we have,
\begin{align*}
&\var_{\theta_n}\big(\Vert s_n\,X_n - \theta_n\Vert^2\big)=\var_{\theta_n}(s_n^2 ||X_n-\theta_n||^2+ 2 \,s_n \, \bar{s}_n \langle X_n -\theta_n, \theta_n \rangle)\\
&\leq 2 \big[s_n^4\, \var_{\theta_n}\big(\Vert X_n - \theta_n\Vert^2\big) + 4 \, s_n^2\, \bar{s}_n^2 \, \var_{\theta_n}(\langle X_n -\theta_n,\theta_n \rangle) \big]\\
&=2 \,n \, s^2_n \, [\,s^2_n + 4 \, \bar{s}_n^2 \,a_n \,] 
\end{align*}
which is obviously less than $O(n)$ if $\bar{s}_n^2 \,a_n =O(1)$.
Otherwise,   
\begin{align*}
\rho_0(\theta_n,s(a_n) X_n)&\geq 2^{-1}\ex_{\theta_n} \log (||s_n X_n - \theta_n||^2/(nr)) \\
&\geq \ex_{\theta_n} \log | \bar{s}_n ||\theta_n|| - s_n \chi_n | /(\sqrt{nr}) \to \infty \text{ as } n \to \infty
\end{align*}
and thus the optimal error in $\mathcal{S}$ is attained at $\text{IL}(\theta_n)$ as defined in Equation~(\ref{ideal.linear.risk}).  

\subsubsection*{Dimension independent bounds}
Here, we produce a dimension independent bound on $g[\that_n^{JS}]$ by explicitly bounding
$L_{\theta_n}(\widehat{\theta}_n)$, $A_{\theta_n}(\that_n,\chat_n)$ and $B_{\theta_n}(\that_n,\chat_n)$.
and then substituting them in Corollary 2.2.\\
By construction $E(\chat_n^{-1})\leq 1$ and by using Lemma~\ref{append-lem-1} we have
$\std_{\theta_n}(c_n^{-1})\leq (n r)^{-1}\, \std_{\theta_n}(\Uhat)$. Now, we have,
\begin{align}
n A_{\theta_n}(\that_n,\chat_n)\big\} \leq & \;\;r^{-2} \, n^{-1} \idealf_{\theta_n}(\that^{JS}_n)\, \var_{\theta_n}(\Uhat) \\
&\;\;+ r^{-3/2} \,n^{-1}\,\std_{\theta_n}(||\that_n-\theta_n||^2/n) \, \std_{\theta_n}(\Uhat)
\end{align}
The ideal flattening coefficient $\idealf_{\theta_n}(\that^{JS}_n)\leq (1+r^{-1})$ and the other quantities
\begin{align}
n B_{\theta_n}(\that_n,\chat_n)\big\} & \leq  r^{-2} \, n^{-1} \, \bias^2(\Uhat_n)\\
n L_{\theta_n}(\widehat{\theta}_n)&\leq r^{-2} \, n^{-1}\,\var_{\theta_n}(||\that_n-\theta_n||^2)
\end{align}
For the JS estimator, for each of the terms in the R.H.S has an upper bound in terms of $n$ and $a_n$.
We can use the following crude upper bounds depending on $n$ only which will  provide us a uniform bound over $\RR^n$:

\subsubsection*{Nature of shrinkage}
The plug-in estimate $\widehat{p}_{\ex}$ performs better than $\widehat{p}_{L}$ and $g[\widehat{\theta}^{\;JS+},\,1+r]$ when $a_n$ is close to $0$ but gets dominated with increasing values of  $a_n$.  And, $g[\widehat{\theta}^{\;JS+}]$ is asymptotically  better than $\widehat{p}_{\ex}$ throughout.
The relationship between $g[\widehat{\theta}^{\;H}]$ and $\widehat{p}_H$ can not be expressed explicitly.
But following \citeasnoun[Theorem 1]{Brown08} we can express the risk of $\widehat{p}_H$ as:
\begin{equation}\label{rel-har}
\rho(\theta_n,\widehat{p}_{H})=\frac{1}{2} \int_{(1+r^{-1})^{-1}}^1 v^{-1} q(\,\theta_n/v,\,\widehat{\theta}^{\,H}\,) \;dv
\end{equation}
where $\that^{\,H}$ denotes the posterior mean of the Harmonic prior. Equation~\ref{rel-har} can be used to numerically evaluate the risk of $\widehat{p}_H$ as the risk of $\widehat{\theta}_H$ has closed form. The fact that these estimators are spherically symmetric will also help. We also get the following crude bound
$$ C \inf_{\beta_n \in A(\theta_n)} q(\,\theta_n/v,\,\widehat{\theta}^{\,H}\,) \leq \rho(\theta_n,\widehat{p}_{H}) \leq C \sup_{\beta_n \in A(\theta_n)} q(\,\theta_n/v,\,\widehat{\theta}^{\,H}\,)$$
where $A(\theta_n)=\big\{\,\beta_n =k \,\theta_n: 1 \leq k \leq \sqrt{1+r^{-1}}\,\big\}$ and $C=\log(1+r^{-1})/2$.


\subsubsection*{Minimaxity over Unrestricted Spaces}
For any dimension $n$, $\widehat{p}_L$ is a minimax estimator. However, in  dimensions greater than 2, $\widehat{p}_L$ is inadmissible and so there exists improved minimax estimators. $\widehat{p}_H$ is an improved minimax estimator than $\widehat{p}_L$ for $n\geq 3$. $g[\widehat{\theta}^{\;JS+}]$ is also an asymptotic minimax estimator and with huge improvements over $\widehat{p}_L$ which can also be explicitly quantified.
Using Theorem \ref{rasl.main.thm} asymptotically minimax predictive density estimates can be constructed around asymptotic minimax location estimates. 

%
%

\subsection{A Motivational Example from Sports Betting}
Consider a game in which the outcomes depend on the actions of $n$ players.
Bets can be placed on a countable collection of (possibly overlapping) measurable sets $\mathcal{A}=\{A_i: i=1,\cdots,k\}$ with $k\leq \infty$ in $\RR^n$.
The maximum growth rate in such a betting market is given by
\begin{align}
\max_{Q\in \mathcal{P}(\RR^n)}\sum_{i=1}^k P(A_i) \log \big \{P(A_i)/Q(A_i) \big \}  
\end{align}
where $P$ is the true probability distribution of the actions in the game, $\mathcal{P}(\RR^n)$ is the set of all probability measures on $\RR^n$ and
$k$ is the cardinality of the collection $\mathcal{A}$.
\par
Assume initially that the collection is exhaustive, i.e,  $\cup_i A_i=\RR^n$. We can construct a mutually disjoint partition $\mathcal{B}=\{B_i: 1\leq i \leq 2^k \}$ of the collection $\mathcal{A}$ where $B_i=\cap_{j=1}^k A_j^{w[i,j]}$ where $w[i,j]$ is the $j^{\text{th}}$ term in the binary expansion of $i$ and for any set $A^0=A^c$ and $A^1=A$. We do not track null $B_i$ in $\mathcal{B}$ and would ignore them through out.  
Let $\kappa(B_i)$ denotes the number of repetitions of the subset $B_i$ in the collection $\mathcal{A}$ 
i.e $\kappa(B_i)=\text{card}\{j: B_i\cap A_j\neq \phi \text{ and } j=1,\cdots,k \}$. Note that $\kappa(B_i) \in [1,k]$
and under finite overlaps we can assume that $\sup_{i=1}^{2^k} \kappa(B_i) = c < \infty$ and we define a weight function on 
$\RR^n$ as $w(x)= \sum_{i=1}^{2^k} c^{-1}\, \kappa(B_i) \, \I_{B_i}(x)$. 
Note that, $w(x) \in (0,1]$ acts as a tilt function for the densities $p(x)$ and $q(x)$.


\begin{thm}
If the probability measure $P$ and $Q$ have densities $p$ and $q$ with respect to Lebesgue measure, then
for any countable collection of exhaustive measurable sets $\mathcal{A}$ we have,
$$\sum_{i=1}^k P(A_i) \log \big \{P(A_i)/Q(A_i) \big \} \leq c \cdot D(p||q)$$
where $D(p||q)=\int p(x) \log \{p(x)/q(x)\} \, dx$ is the differential relative entropy between $P$ and $Q$.
\end{thm}
\begin{proof}
If the collection consists of mutually disjoint sets then the proof follows from the data processing inequalities associated with quantization idea in information theory. The function $t\log t$ is strictly convex if $t>0$. 
So for any positive random variable $T$ and any sigma-finite measure, by Jensen's inequality we have,
$E_{\mu}(T \log T) \geq E_{\mu}(T) \log E_{\mu}(T)$. 
For any measurable set $A$, with $T(x)=p(x)/q(x)$ and measure $\mu(x)=q(x)/Q(A)\, dx$ we have 
$P(A) \log P(A)/Q(A) \leq \int_A p(x) \log \{p(x)/q(x)\} \, dx$ and so the proof extends to mutually exclusive cases. 
\par
If the events are not mutually disjoints then we can construct its mutually disjoint partition $\mathcal{B}=\{B_i: 1\leq i \leq 2^k \}$ as above and using the Log-Sum inequality \cite[Theorem 2.7.1]{Cover-book} separately on each $A_i$ we have,
$$ \sum_{i=1}^k P(A_i) \log \{P(A_i)/Q(A_i)\} \leq \sum_{i=1}^{2^k} \kappa(B_i) \, P(B_i) \log \{P(B_i)/Q(B_i)\}$$ 
and again using the above quantization argument we can show that the R.H.S above is less than $c \int w(x) p(x) \log \{p(x)/q(x)\} \, dx \stackrel{\Delta}{=} c \, D(w.p||w.q)$. Now, observe that
\begin{align*}
D(w.p||w.q)-D(p||q)&= \int \big(1-w(x)\big)\, p(x) \log \big \{q(x)/p(x)\big\} \, dx \\
&\leq \log \bigg[ \int \big(1-w(x)\big)\, q(x) \, dx\bigg]
\end{align*}
by Jensen's inequality and the result follows as $\int \big(1-w(x)\big)\, q(x) \, dx \leq 1$.
\end{proof}
If the collection is not exhaustive we can restrict our densities to the corresponding subsets of $\RR^n$.
\subsubsection{An illustration with a Dataset} We consider the Baseball data  that was used to show the advantage of  shrinking location estimates in \citeasnoun{Efron77}. The dataset consists of $18$ players (so, $n=18$ which is not so high dimensions) with exactly $45$ at-bats on a particular date during the $1970$ season. The objective is to predict the performance of the players on the remainder of the season .
\paragraph{}
The number of hits $(H)$ and the number of at-bats $(N)$ over two  portions of the season were  
$$H_{ji} \stackrel{ind.}{\sim} \text{Binomial} ( N_{ji},p_i) ,  \quad j=1,2 ; \quad i=1, \ldots, n .$$
Where $j=1$ denotes past data and $j=2$ represents the unknown future. 
As the variance of the Binomial model depends of the mean parameter $p_i$, a variance stabilization transformation \cite{Brown08a} is conducted (which goes through as $N_{ij}$ are quite large). 
The transformation 
\begin{align} X_{ji} = \arcsin \bigg(\frac{H_{ji} + 1/4}{ N_{ji} + 1/2}\bigg)^{1/2} \end{align}
reduces the binomial model to the normal model
\begin{align}
 X_{ji} \sim N(\theta_i, \sigma_{ji}^2) \; \text{ where } \theta_i = \arcsin \sqrt{\,p_i}\,,\; \sigma^2_{ij} = \big(4 N_{ji}\big)^{-1} 
\end{align}
and $X_{.i}$ independent for $1 \leq i \leq n$.
With the past $P=X_{1.}$ and the future $F=X_{2.}$ we have the following predictive set-up :  
\begin{align} 
F|\theta_n \sim  N\big(\theta_n,v_y I_n\big) \;;\quad \;  P|\theta_n \sim  N(\theta_n,v_x I_n).
\end{align}
We want joint predictive densities of the future performances of players in this standardized model.
We use a very naive evaluation strategy by considering the entire season's batting average as the true parametric value.
In the entire season the players ended up playing around $400$ games on the average. So, evaluating the predictive densities  at
$ \theta_i^0 = \arcsin \big(p_i^{\text{full}}\big)^{1/2}$ where $p_i^{\text{full}}$ are the batting averages from the entire season will not be terrible. Evaluation procedures with guarantees may be developed in a sequential set-up \cite{Lai11}.
\begin{table}
\begin{tabular}[c]{|r|r|r|r|r|r|r|}
\hline & & & & & & \\[-1ex]
$r$ & $\widehat{p}_{E}$ & $\widehat{p}_{L}$ & $g_M$ &    $g[\widehat{\theta}^{\;JS+}]$ & $g[\widehat{\theta}^{\;H}]$&$\widehat{p}_{H}$\\
 & & & & & & \\[-1ex]
\hline & & & & & & \\[-1ex]
0.1 & 22.963 & 19.451 & 15.487 & 11.435 & 19.232 & 19.578\\
0.2 & 11.482 & 14.174 & 10.539 & 7.418 & 13.982 & 14.289\\
0.5 & 4.593 & 8.326 & 5.418 & 3.717 & 8.188 & 8.424\\[1ex]
\hline & & & & & & \\[-1ex]
1 & 2.296 & 5.067 & 2.886 & 2.047 & 4.975 & 5.142 \\
 & & & & & & \\[-1ex]
\hline & & & & & & \\[-1ex]
2 &1.148 &2.868 &1.415 &1.081 &2.815 &2.924\\
5 &0.459 &1.250 &0.524 &0.448 &1.227 &1.286 \\
10 &0.23 &0.645 &0.248 &0.227 &0.633 &0.614\\[1ex]
\hline\hline
\end{tabular}
\caption{Predictive loss of different Gaussian strategies on the Baseball data.}
\label{table-baseball}
\end{table}
\paragraph{}
While using shrinkage on the location estimators we shrink towards the grand average. 
We evaluate the $6$ different predictive strategies of Section~\ref{oracle-ineql} for different values of the future to past variability.
The value of $r$ will be close to $0.1$ when we consider prediction on the entire remaining half of the season.
\par
We find that for any choice $r$, $g[\widehat{\theta}^{\;JS+}]$ is the best one among the $6$ estimators considered.
Also,  $d_n^2-d_n^1$ (as discussed in Section~\ref{oracle-ineql}) is decreasing in $r$.
$\widehat{p}_E$ behaves well when $r$ is large and horribly for small values.
The losses for $\widehat{p}_{H}$ and  $g[\widehat{\theta}^{\;H}]$ are very similar.
 

\section{Restricted minimax predictive risk of $\mathcal{G}[p]$}\ssskip
A typical member in the class $\mathcal{G}[p]$ of all product Gaussian predictive densities is represented by $g[\that_n,\Dhat_n]=\prod_{i=1}^n N(\that(i), \widehat{d}(i)\,\,\sigma_f^2)$.   Generalizing the argument in Lemma~\ref{lem-lowbound} we see that a lower bound on the minimum predictive risk $\rho_{\,p\,}(\theta_n,\that_n)$ of all density estimates in $\mathcal{G}_n[p]$ that have mean $\that_n$, is given by
\begin{align}\label{eqn.l.bound}
\rho_{\,p\,}(\theta_n,\that_n)\geq \frac{1}{2} \sum_{i=1}^n \ex_{\,\theta(i)} \big\{\log\big(1+r^{-1}\, (\that(i) - \theta(i))^2\big)\big\}.
\end{align}
The predictive risk of the estimate $g[\that_n,\Dhat_n]$ is given by
\begin{align}
2 \rho(\theta_n,g[\that_n,\Dhat_n]) &= \sum_{i=1}^n \ex_{\,\theta_n} \log(\widehat{d}(i)) + \ex_{\,\theta_n} \sum_{i=1}^n \bigg\{ \frac{1+ (\that(i) - \theta(i))^2) - \widehat{d}(i)}{\widehat{d}(i) }\bigg\}. 
\end{align}
It is not necessarily true that 
$$\rho_{\,p\,}(\theta_n,\that_n)=\min_{\Dhat_n \in \RR_+^{n}}\;\rho(\theta_n,g[\that_n,\Dhat_n])$$
asymptotically equals the lower bound given in Equation~(\ref{eqn.l.bound}). In the previous section we saw that under sufficient regularity conditions these bounds matches. The ideas there can be extended to block-wise estimators  and to non-orthogonal models by using the concept of Mallow's unbiased risk estimates.
In the $\ell_0$ sparse predictive space as the degree of sparsity tends to zero, i.e., $s/n \to 0$ as $n \to \infty$, the lower bound given in Equation~(\ref{eqn.l.bound}) is significantly greater than the minimax predictive risk over $\mathcal{G}[p]$. And so, procedure used in the previous section can not be used for finding the asymptotic minimax predictive Gaussian risk over $\Theta(n,s)$.
\subsection*{Minimax predictive risk over sparse parameter spaces}
Here we outline the proof of Theorem~\ref{thm.minimax.risk}. Following the Bayes-Minimax procedure of \citet{Johnstone-book} (Chapter 4.4) the multivariate minimax problem can be reduced to univariate minimax problem with moment prior constraints
$$\mathfrak{m}(\eta)=\{\pi \in \mathcal{P}(\RR): \pi(0)\geq 1- \eta\}$$
where $\mathcal{P}(\RR)$ is the collection of all probability measures on $\RR$. In Theorem~1.1 in \cite{l0-sparse} we have the univariate minimax risk 
$$\min_{\phat} \max_{\pi \in \mathfrak{m}(\eta)} \int \rho(\theta,\phat) \, \pi(\theta) \, d\theta \sim (1+r)^{-1} \,\eta \,\log \eta^{-1} \text{ as } \eta \to 0. $$
When restricted to the Gaussian family the minimax risk will be
$$\min_{\phat \in \mathcal{G}} \max_{\pi \in \mathfrak{m}(\eta)} \int \rho(\theta,\phat) \, \pi(\theta) \, d\theta \sim f(\eta) \text{ as } \eta \to 0 $$
where $f(\eta)=r^{-1} \,\eta \,\log \eta^{-1}$. In this univariate asymptotic set-up the lower bound in Equation~(\ref{eqn.l.bound}) is much lower than the asymptotic rate $\eta \log \eta^{-1}$ and hence unusable. We get an upper bound on the minimax Gaussian risk as from point estimation theory \cite{Donoho94b}it follows that the minimax plug-in risk in this asymptotic set-up is $f(\eta)$.
For a lower bound consider the predictive risk of the normal density estimate $g[\that,\dhat]$  
\begin{align}
\rho\big(\,\theta,g[\that,\dhat]\,\big)=\ex_{\,\theta}\big(\,\log \dhat\;\big) + \ex_{\theta}\big\{\,\dhat^{\,-1}\cdot(1+(\that-\theta)^2)-1\big\}.
\end{align}
And the idea is to establish the necessity of threshold zone as done in \citet{Johnstone04}.
For $\rho\big(0,g[\that,\dhat]\,\big)$ -- the predictive risk of $g[\that,\dhat]$ at the origin, to be lower than the order of $\eta$ we need a threshold size of at least $\lambda(\eta)=\sqrt{2 \log \eta^{-1}}.$
And for density estimators of the form 
\begin{align}
\phat\;[\,\lambda(\eta)\,]\,(\cdot|X)=\left \{ \begin{array}{cc} 
N\big(\;0,\;\sigma_f^2\;\big) & \text{ if } |X| \leq \lambda(\eta) \\
N\big(\;\that(X),\;\dhat(X) \, \sigma_f^2\;\big) & \text{ if } |X| > \lambda(\eta) 
\end{array}\right.
\end{align}
the supremum predictive risk at the non-zero support points is $f(\eta)$, i.e.,
\begin{align}
\sup_{\theta \neq 0}\;\;\rho\big(\,\theta,\;\phat\;[\,\lambda(\eta)\,]\,\big) \sim f(\eta) \text{ as } \eta \to 0 .
\end{align}
Thus, it follows that sup-optimality of the class $\mathcal{G}[p]$ is $1+r^{-1}$.


\section*{Acknowledgements}\ssskip
G.M. is indebted to Thomas Cover for the numerous enjoyable conversations, and his insights on information complexity and density estimation in predictive context, which inspired and reshaped many ideas in the paper.
\appendix
\section{}
\begin{lem}\label{lem.k1}
$Y_n$ is sequence of random variables such that $Y_n \stackrel{d}{=} \chi^2_n(\lambda_n)$ for a non-negative and increasing sequence $\{ \lambda_n: n\geq 1\}$ then for $ n \geq 5 $ we have
$$\var\big(\,Y_n^{-1} \,\big)\leq k_1(n) \cdot n^{-3} \text{ where } k_1(n) = 3 \,(1-2/n)^{-2}(1-4/n)^{-1}.$$
\end{lem}

\begin{proof}
We observe that $Y_n$ being a non-central chi-square random variable can be  written as convolution of  central Chi-square and Poisson random variables 
$$Y_n \stackrel{d}{=} \chi^2_{n+2N} \text{ where } N_n \stackrel{d}{=} \text{Poisson}(\lambda_n/2).$$ 
Decomposing the variance by conditioning on the Poisson random variable we have, 
\begin{align*}
\var \left(Y_n^{-1}\right)&=\var_{\lambda_n}\bigg(E\left( Y_n^{-1} \vert N_n \right)\bigg)
+E_{\lambda_n}\bigg(\var\left(Y_n^{-1}|N_n\right)\bigg)\\
&=\var_{\lambda_n}\left(\frac{1}{n+2N_n-2}\right)+E_{\lambda_n}\left(\frac{2}{(n+2N_n-2)^2(n+2N_n-4)}\right)
\end{align*}
which follows from moments of central chi-square (gamma) distribution and 
as $N_n\geq 0$  the second term on the R.H.S is  $ \leq 2 (n-2)^{-2}(n-4)^{-1}$ and by Lemma~\ref{append-lem-1} we have
\begin{align*}
(n-2)^2\var_{\lambda_n}\left(\frac{1}{n+2N_n-2}\right) &= \var_{\lambda_n}\left(\frac{1}{1+2N_n/(n-2)}\right) \\
&\leq \big\{1+2 \,E(N_n)/(n-2)\big\}^{-4}\;\var \left(\frac{2\,N_n}{n-2}\right)\\
&=\frac{4 \,\lambda_n \,(n-2)^2}{(n-2+2\lambda_n)^4}\leq \frac{1}{2(n-2)}.
\end{align*}
Thus, $\var \left(Y_n^{-1}\right) \leq 3 \,(n-2)^{-2}(n-4)^{-1}$.
\end{proof}

\begin{lem}
If $Y_n \stackrel{d}{=} \chi^2_n(\lambda_n)$ and $\lambda_n$ is an increasing sequence then
$$\lambda_n^2 \, P(Y_n \leq n-2) \leq O(n)$$
\end{lem}
\begin{proof}
Holds trivially for $\lambda_n \leq O(\sqrt{n})$. So we will prove for all other sequences i.e sequence where $\lambda_n/\sqrt{n}$ is not bounded. 
Note that
$P(Y_n \leq n-2)\leq P(Y_n \leq n).$ 
And as $Y_n$ is a non-central chi-square we have
$$Y_n\stackrel{d}{=} V_{n+2N}\text{ where } N \stackrel{d}{=} \text{ Poisson} (\lambda_n) \text{ and } V_{n}\stackrel{d}{=}\chi^2_{n}(0)$$
Now, for any fixed $n$ and $N$ we have, 
$$P(V_{n+2N}\leq n) \leq 2 P(V_{m+2N}\leq m) \text{ for all } m \geq n \text{ such that } m - n \text{ is large}.$$
Because $P(V_{m+2N}\leq m \vert V_{n+2N}\leq n)\leq P(\chi^2_{m-n}(0)\leq m-n) \leq 1/2$.
So,
\begin{align*}
\lim_{n \rightarrow \infty} P(Y_n \leq n) &= \lim_{n \rightarrow \infty} E_{\lambda_n} \bigg\{P\bigg(V_{n+2N} \leq n \bigg\vert N\bigg)\bigg\}\\
&\leq 2 \lim_{n \rightarrow \infty} E_{\lambda_n} \bigg\{\lim_{n \rightarrow \infty}P\bigg(V_{n+2N} \leq n \bigg\vert N\bigg) \bigg\}\\
&= 2 \lim_{n \rightarrow \infty} E_{\lambda_n} \left\{\Phi\left(\frac{-2N}{\sqrt{2n+4N}}\right)\right\}
\end{align*}
as we can do a normal approximation to the sequence of central chi-square random variables.
Next, we interchange the integrals (by Fubini's as integrand is positive) and then use bounded convergence theorem to have,
\begin{align*}
\lim_{n \rightarrow \infty} P(Y_n \leq n)& \leq 
2 \lim_{n \rightarrow \infty} \int \phi(z) P_{\lambda_n}\left(\frac{2N}{\sqrt{2n+4N} }\leq z \right ) dz\\
&= 2  \int \phi(z) \lim_{n \rightarrow \infty} P_{\lambda_n}\left(\frac{2N}{\sqrt{2n+4N} }\leq z \right ) dz
\end{align*}
Now for all large $n$, $\lambda_n$ is large (as $\lambda_n$ increasing and $\lambda_n/\sqrt{n}$ is not a bounded sequence). So each large $n$, we can separately do a normal approximation to the Poisson random variable $N$.
\\
Consider the case first when $\lambda_n > O(n)$. In this case the following naive bound will work:
$$ P_{\lambda_n}\left(\frac{2N}{\sqrt{2n+4N} }\leq z \right ) \leq  P_{\lambda_n}\left(\frac{\sqrt{N}}{\sqrt{n} }\leq z \right )\sim \widetilde{\Phi}\left(\frac{\lambda_n-n z^2}{\sqrt{\lambda_n}}\right).$$
We will use this bound for all $z$ such that $z^2 \leq t_n$ where $t_n$ equals $n^{-1}(\lambda_n-\sqrt{\lambda_n}\sqrt{4\log \lambda_n + 2 \log  n})$.
Also note that,
$$ \lambda_n^2 \widetilde{\Phi}\left(\frac{\lambda_n-n z^2}{\sqrt{\lambda_n}}\right) \leq O(n) \text{ for all }  z^2 \leq t_n \text{ and } \widetilde{\Phi}(t_n)=O(n\lambda_n^{-2}).$$
And so,  it follows that $\lambda_n^2\lim_{n \rightarrow \infty} P(Y_n \leq n)\leq O(n)$. 
\end{proof}

\begin{lem}\label{append-lem-1}
For any non-negative random variable $Y$
$$\var \big\{(1+Y)^{-1}\big\}\leq \{1+E(Y)\}^{-4}\; \var\big(Y\big).$$
\end{lem}

\begin{proof} As $Y$ is non-negative we have
$$\left ( \frac{1}{1+Y}-\frac{1}{1+E(Y)}\right )^2=\frac{(Y-E(Y))^2}{(1+Y)^2 (1+EY)^2}\leq \frac{( Y-E(Y) )^2}{(1+EY)^2}.$$
Now, taking expectation on both sides and using Bias-Variance decomposition we get
$$\var\left(\frac{1}{1+Y}\right) + \left ( E\left ( \frac{1}{1+Y}\right)-\frac{1}{1+E(Y)}\right )^2\leq \{1+E(Y)\}^{-4}\;\var( Y ).$$
This completes the proof.
\end{proof}

\begin{lem}\label{append-lem-2}
For any random variable $X$ we have $\var(X_+)\leq \var(X)$. 
\end{lem}

\begin{proof} With the decomposition of $X=X_+ - X_-$ we have
\begin{align*}
\text{Var}(X)&=\E(X^2)-\E^2(X)\\
&=\E(X^2_+) + \E(X^2_-)-\E^2(X_+)-\E^2(X_-)+2\E(X_+)\,\E(X_-) \\
&=\text{Var}(X_+)+\text{Var}(X_-)+2 \E(X_+) \E(X_-)
\end{align*}
and we get the stated result as all the terms in R.H.S. are non-negative.
\end{proof}

\bibliography{pred-inf}
\end{document}